\documentclass[a4paper,11pt,twoside]{amsart}

\usepackage{amsmath, amsthm, amsfonts, amssymb}

\usepackage[colorlinks=true]{hyperref}

\usepackage{mathabx}

\newcommand{\R}{\ensuremath{\mathbb{R}}}

\newcommand{\N}{\ensuremath{\mathbb{N}}}

\renewcommand{\div}{\textrm{div }}
\newcommand{\curl}{\textrm{curl }}

\newcommand{\eps}{\varepsilon}

\theoremstyle{definition} 
\newtheorem{theorem}{Theorem}[section]

\newtheorem{definition}[theorem]{Definition}

\newtheorem{lemma}[theorem]{Lemma}
\newtheorem{remark}[theorem]{Remark}

\begin{document}

\title[Piecewise constant subsolutions for Muskat]{Piecewise constant subsolutions for the Muskat problem}

\author{Clemens F\"{o}rster}
\address{Institut f\"ur Mathematik, Universit\"at Leipzig, D-04103 Leipzig, Germany}
\email{clemens.foerster@math.uni-leipzig.de}

\author{L\'aszl\'o Sz\'ekelyhidi Jr.}
\address{Institut f\"ur Mathematik, Universit\"at Leipzig, D-04103 Leipzig, Germany}
\email{laszlo.szekelyhidi@math.uni-leipzig.de}

\date{\today}

\begin{abstract}
We show the existence of infinitely many admissible weak solutions for the incompressible porous media equations for all Muskat-type initial data with $C^{3,\alpha}$-regularity of the interface in the unstable regime and for all non-horizontal data with $C^{3,\alpha}$-regularity in the stable regime. Our approach involves constructing admissible subsolutions with piecewise constant densities. This allows us to give a rather short proof where it suffices to calculate the velocity and acceleration at time zero - thus emphasizing the instantaneous nature of non-uniqueness due to discontinuities in the initial data. 
\end{abstract}

\maketitle

\section{Introduction}

We consider the evolution of two incompressible fluids with the same viscosity and different densities, moving in a porous two-dimensional medium with constant permeability under the action of gravity according to Darcy's law. After non-dimensionalizing, the equations describing the evolution of density $\rho$ and velocity $u$ are given by (see \cite{otto:ipm,cg:condynipm} and references therein)
\begin{align}
\partial_t \rho + \div(\rho u) &= 0\,,\label{IPM1} \\
\div u &= 0\,,\label{IPM2}\\
u + \nabla p &= -(0,\rho)\,,\label{IPM3} \\
\rho(x,0) &= \rho_0(x)\,.\label{IPM4}
\end{align}
We assume that at the initial time the two fluids, with densities $\rho^+$ and $\rho^-$, are separated by an interface which can be written as the graph of a function over the horizontal axis. That is, 
\begin{align}\label{init}
\rho_0(x) = \begin{cases} \rho^+ & x_2 > z_0(x_1),\\ \rho^- & x_2 < z_0(x_1). \end{cases}
\end{align}
Thus, the interface separating the two fluids at the initial time is given by $\Gamma := \{(s,z_0(s))| s \in \R\}$. We distinguish the following cases: If $\rho^+ > \rho^-$, which means that the heavier fluid is on top, we speak of the unstable regime. The case $\rho^+< \rho^-$ is called the stable regime. 

\subsection*{The Muskat problem}

Since for given $\rho(x,t)$ at a fixed time $t$, $u$ is the solution of an elliptic problem by the Biot-Savart law, the equations \eqref{IPM1}-\eqref{IPM3} describe the evolution of the density in time. Assuming that $\rho(x,t)$ remains in the form \eqref{init} for positive times, the system reduces to a non-local evolution problem for the interface $\Gamma$, known as the Muskat problem.
If the sheet can be presented as a graph as above, one can show (see for example \cite{cg:condynipm}) that the equation for $z(s,t)$ is given by
\begin{align}\label{evolz}
\partial_t z(s,t) = \frac{\rho^- - \rho^+}{2\pi} \int_{- \infty}^{\infty}  \frac{(\partial_{s} z(s,t) - \partial_{s} z(\xi,t))(s-\xi)}{(s-\xi)^2 + (z(s,t) - z(\xi,t))^2} d\xi.
\end{align}
%The same equation can be obtained for the description of the flow between closely spaced parallel sheets of glass, a so-called Hele-Shaw cell \cite{}. This can be viewed as a simplified experimental model for the flow in a porous medium.
The behavior of solutions of \eqref{evolz} depends strongly on the sign of $\rho^- - \rho^+$. In the stable case, this equation is locally well-posed in $H^3(\R)$, see \cite{cg:condynipm} or \cite{cgvs:gr2dm} for an improved regularity. However, in the unstable regime, we have an ill-posed problem, see \cite{Saffman:1958cu,cg:condynipm}. In particular, in the unstable case, there are no general existence results for \eqref{evolz} known.

Thus, the description of \eqref{IPM1}-\eqref{IPM4} as a free boundary problem seems not suitable for the unstable regime.
In \cite{otto:ipm,otto:ema} F.~Otto used a Lagrangian relaxational approach in the spirit of optimal transportation and he proved the existence of a unique relaxation limit $\bar{\rho}$ in the case of the flat initial datum $z_0 \equiv 0$. Whilst Otto's relaxation limit does not satisfy the original system \eqref{IPM1}-\eqref{IPM4} any more, the relaxed density can be thought of as a macroscopic average of an infinitely fine mixture of the two phases $\rho^{\pm}$. More precisely, there is a growing \emph{mixing zone} around the initial unstable sheet $\Gamma$, where the two densities are mixed with average density $\bar{\rho}(x,t)$ which satisfies an evolution equation (a variant of the 1D Burgers equation). Such a behaviour is reminiscent of the physically expected behaviour in the unstable regime \cite{Wooding:1976vca,otto:ipm}.

\subsection*{Weak solutions}

Taking the curl of \eqref{IPM3}, we can eliminate the pressure and obtain $\curl u = -\partial_{x_1}\rho$. This motivates the definition of weak solutions in the following form.
\begin{definition}\label{d:weak}
Let $\rho_0 \in L^{\infty}(\R^2)$ and $T>0$. We call $(\rho,u) \in L^{\infty}(\R^2 \times [0,T))$ a weak solution of \eqref{IPM1}-\eqref{IPM4} with initial data $\rho_0$ if
\begin{align*}
\int_0^T \int_{\R^2} \rho (\partial_t \phi + u \cdot \nabla \phi) dx dt &= \int_{\R^2} \phi(x,0) \rho_0(x)dx \quad \forall \phi \in C^{\infty}_c([0,T) \times \R^2)\\
\int_{\R^2} u \cdot \nabla \phi dx &= 0 \quad \forall \phi \in C^{\infty}_c(\R^2)\\
\int_{\R^2} (u+(0,\rho)) \cdot \nabla^{\perp} \phi dx &= 0 \quad \forall \phi \in C^{\infty}_c(\R^2).
\end{align*}
\end{definition}

In \cite{cfg:luipm} infinitely many weak solutions of \eqref{IPM1}-\eqref{IPM4} were constructed to any initial datum $\rho_0$. The construction is a variant of convex integration, as used in \cite{dlseuler,dls:hprinc}, and in particular the result in \cite{cfg:luipm} is in a sense the IPM-analogue of the Scheffer-Shnirelman construction for the Euler equations \cite{Scheffer93,Shnirelman,w:existweak}. However, these weak solutions do not retain the geometric structure of initial data of the type \eqref{init}, and in particular the density $\rho$ may exceed the initial densities $\rho^{\pm}$. 

\subsection*{Admissibility and mixing solutions}

Motivated by the analogous development of admissible weak solutions for the Euler equations \cite{dls:admcrit} as well as the result of Otto in \cite{otto:ipm}, in \cite{sze:ipm} \emph{admissible weak solutions} were introduced by the second author as weak solutions $(\rho,u)$ such that 
\begin{equation}\label{e:admissible}
\rho(x,t) \in [\rho^-,\rho^+]\quad\textrm{ for a.e. }(x,t),
\end{equation}
(or $\rho \in [\rho^+,\rho^-]$ in the stable case, where $\rho^+<\rho^-$). 

In \cite{sze:ipm} the second author showed that there exist infinitely many \emph{admissible weak solutions} for the Muskat problem in the unstable regime with flat initial data by the convex integration method. Moreover, an interesting connection between such admissible weak solutions and the relaxation limit of Otto is provided by the concept of subsolution (see below in Definition \ref{d:subsol}). In a nutshell, weak solutions constructed by convex integration arise by adding high-frequency spatially localized perturbations to an initial $\bar{\rho}(x,t)$, which can be thought of as an averaged density: by increasing the frequency of the perturbations, one can easily construct a sequence of admissible weak solutions $(\rho_k,u_k)$, such that $\rho_k \overset{\ast}{\rightharpoonup} \overline{\rho}$ in $L^{\infty}$, see \cite{sze:ipm}. Whereas the construction of weak solutions from strict subsolutions is by now very well understood in the general setting, constructing strict subsolutions to the initial value problem still requires ad-hoc methods \cite{sze:vortex,sze:ipm,bszw:nonuniq,cdk:gip} and no general technique seems to exist.

Recently, the result from \cite{sze:ipm} was generalized to arbitrary initial curves $z_0$ by Castro, Cordoba and Faraco in \cite{ccf:ipm}.
The main theorem in \cite{ccf:ipm} states that for each $z_0 \in H^5(\R)$ and for $\rho^+ > \rho^-$, there exist infinitely many admissible weak solutions to \eqref{IPM1}-\eqref{IPM4} with initial data \eqref{init}. The key point in the proof is to show the existence of certain strict subsolutions, which are in a sense the geometrically nonlinear analogues of the subsolution constructed in \cite{sze:ipm}: the two ingredients defining the subsolution are a density and an evolving sheet (as in the original Muskat problem), whose translates are then level-sets of the density (see also Section \ref{s:sub} below). The density is chosen exactly as in \cite{sze:ipm}, but the evolving sheet needs to solve a nonlinear and nonlocal evolution equation $\partial_tz=\mathcal{F}(u)$ (see (1.11)-(1.12) in \cite{ccf:ipm}) -- the analysis of this equation is the central part of the proof in \cite{ccf:ipm}.

\subsection*{The main result}

The aim of this paper is to give an alternative and considerably simpler proof of the main result from \cite{ccf:ipm} for the unstable case. The key difference is that we allow the density to be piecewise constant -- in turn, rather than having to prove local well-posedness for a non-linear and non-local evolution equation, it suffices to obtain expressions for the velocity and acceleration of a double-sheet at time $t=0$. Our construction is similar in spirit to fan-subsolutions, introduced for flat shock-waves  for the compressible Euler equations in \cite{cdk:gip}. The advantage is not only the considerably shorter proof, but also the lower regularity requirement on the initial curve: we require $C^{3,\alpha}$ with decay at infinity, in contrast to $H^5$ in \cite{cdk:gip}) . Finally, we extend our result also to the stable case, provided the initial interface is not horizontal flat (thus, extending the observation made in \cite{ccf:ipm} Section 7 concerning flat non-horizontal interfaces in the stable case). 

Our assumption on the initial datum is that the initial interface is asymptotically flat with some given slope $\beta\in\R$, i.e. $\rho_0$ is given by \eqref{init} with
\begin{equation}\label{e:z0}
	z_0(s) = \beta s + \overline{z}_0(s)
\end{equation}
for some $\overline{z}_0$ with sufficiently fast decay at infinity. More precisely, for any $\alpha>0$ define the seminorm
$$
[f]^*_{\alpha}:=\sup_{|\xi|\leq 1,s\in\R}(1+|s|^{1+\alpha})\frac{|f(s-\xi)-f(s)|}{|\xi|^{\alpha}},
$$
and for any $k\in\N$ the norm
$$
\|f\|_{k,\alpha}^*:=\sup_{s\in\R,j\leq k}(1+|s|^{1+\alpha})|\partial_s^jf(s)|+[\partial_s^kf]_{\alpha}^*.
$$
We denote by $C_*^{k,\alpha}(\R):=\{f\in C^{k,\alpha}(\R):\,\|f\|_{k,\alpha}^*<\infty\}$.

\begin{theorem}\label{Mth}

Let $z_0(s) = \beta s + \overline{z}_0(s)$ with $\overline{z}_0 \in C^{3,\alpha}_*(\R)$ for some $0<\alpha<1$ and $\beta \in \R$.
\begin{enumerate}
\item[(i)] In the unstable case $\rho^+ > \rho^-$, for each $\beta \in \R$, there exists $T_*>0$ such that there exist infinitely many admissible weak solutions to \eqref{IPM1}-\eqref{IPM4} in $[0,T_*)$.
\item[(ii)] In the stable case $\rho^+ < \rho^-$, whenever $\beta\neq 0$ and $\|\partial_s\bar{z}_0\|_{L^\infty}<|\beta|$ there exists $T_*>0$ such that there exist infinitely many admissible weak solutions to \eqref{IPM1}-\eqref{IPM4} in $[0,T_*)$.
\end{enumerate}
\end{theorem}

As pointed out above, the advantage of our method is the simplicity of the proof and the lower regularity requirement for the initial curve. However, this comes at a small price: as will be explained below in Section \ref{s:sub}, the admissible weak solutions obtained in our Theorem (as well as those obtained in \cite{sze:ipm} and \cite{ccf:ipm}) have the common feature, that there is an expanding \emph{mixing zone} $\Omega_{mix}(t)$ concentrating on the initial interface $\Gamma_0$ at time $t=0$, where the two fluids are ``infinitely mixed''. The rate of expansion of the mixing zone in the unstable case is given by 
$$
\frac{\rho^+-\rho^-}{2}c
$$
for some $c>0$. The constructions in \cite{sze:ipm} and \cite{ccf:ipm} admit any $c<2$, and indeed, $c_{max}=2$ seems the maximal expansion rate possible (see \cite{sze:ipm} and the discussion following Theorem \ref{t:mixing} below). In contrast, our construction admits only $c<1$. However, this is only a problem for the simplest possible choice of piecewise constant density in \eqref{e:dens} - we will show in Section \ref{s:general} that with a more general piecewise constant density any expansion rate $c<2$ is reachable. 

We also point out that, just like in \cite{ccf:ipm}, our result is local in time: there is a short time of existence $[0,T_*]$; moreover $T_*\to 0$ as we reach the maximal speed $c\to c_{max}$. This is at variance with the result in \cite{sze:ipm}, which is global in time.

The paper is organized as follows. In Section \ref{s:sub} we recall the notion of a subsolution for \eqref{IPM1}-\eqref{IPM4} and show in Theorem \ref{Subs} that under appropriate estimates for the interface $z(s,t)$ as time $t\to 0$, we are able to construct a subsolution and thus prove our main result Theorem \ref{Mth}. In Section \ref{s:u} we recall the expression for the normal component of the velocity obtained by the Biot-Savart law for piecewise continuous densities, and provide Schauder-type estimates for the associated integral operators. Section \ref{s:z} is devoted to the construction of the interface curve $z(s,t)$ by evaluating the velocity and symmetrized acceleration at time $t=0$. Finally we generalize these results to more general piecewise constant densities in Section \ref{s:general}.

\subsection*{Acknowledgments} 
The authors gratefully acknowledge the support of the ERC Grant Agreement No. 724298.

%%%%%%%%%%%%%%%%%%%%%%%%%%%%%%%%%%%%%%%%%%%%%%%%%%%%%%%%%%%%%%%%%%%%%%%%%%%%

%%%%%%%%%%%%%%%%%%%%%%%%%%%%%%%%%%%%%%%%%%%%%%%%%%%%%%%%%%%%%%%%%%%%%%%%%%%%

\section{Subsolutions for the IPM Equations}\label{s:sub}

We start by recalling the general strategy for the construction of weak solutions, as followed in \cite{ccf:ipm} and \cite{sze:ipm}. The basic idea is to construct a suitable \emph{admissible subsolution}, and then apply the general machinery of convex integration.  

Observe that if $(\rho,u)$ is a solution of \eqref{IPM1}-\eqref{IPM4}, then so is $(\tilde\rho,\tilde u)$ given by
\begin{equation}\label{e:rescale}
\tilde\rho(x,t)=a\rho(x,at)+b,\quad \tilde u(x,t)=a u(x,at).
\end{equation} 
Then, by choosing 
$$
a=\frac{\rho^+-\rho^-}{2},\quad b=\frac{\rho^++\rho^-}{2}
$$
we may assume that the Muskat-type initial datum \eqref{init} is given by $\rho^{\pm} = \pm1$ (in the stable case the signs are obviously swapped). Under this normalization, admissibility amounts to the requirement
$$
|\rho|\leq 1\textrm{ for a.e. }(x,t).
$$

\begin{definition}\label{d:subsol}
Let $T > 0$. We call a triple $(\rho,u,m) \in L^{\infty}(\R^2 \times [0,T))$ an admissible subsolution of \eqref{IPM1}-\eqref{IPM4} if there exist open domains $\Omega^{\pm}, \Omega_{mix}$ with $\overline{\Omega^+} \cup \overline{\Omega^-} \cup \Omega_{mix} = \R^2\times [0,T)$ such that
\begin{enumerate}
\item[(i)] The system
\begin{equation}\label{IPMsub}
\begin{split}
\partial_t \rho + \div m &= 0\\
\div  u &= 0\\
\curl  u &= -\partial_{x_1} \rho\\
\rho\vert_{t=0} &= \rho_0
\end{split}
\end{equation}
holds in the sense of distributions in $\R^2 \times [0,T)$;
\item[(ii)] The pointwise inequality
\begin{align}\label{str}
\left|m-\rho u + \frac{1}{2} (0,1-\rho^2) \right| \leq  \frac{1}{2} \left(0,1-\rho^2 \right),
\end{align}
holds almost everywhere;
\item[(iii)] $|\rho(x,t)| = 1$ in $\Omega^{+}\cup\Omega^-$;
\item[(iv)] In $\Omega_{mix}$ the triple $(\rho,u,m)$ is continuous and \eqref{str} holds with a strict inequality.
\end{enumerate}
\end{definition}

Admissible subsolutions lead to the existence of infinitely many admissible weak solutions by the Baire category method for convex integration, see for instance the Appendix in \cite{sze:ipm}. As pointed out in \cite{ccf:ipm}, a slight modification of the general technique leads to the following statement:
\begin{theorem}\label{t:mixing}
Suppose there exists an admissible subsolution $(\bar{\rho},\bar{u},\bar{m})$ to \eqref{IPM1}-\eqref{IPM4} . Then there exist infinitely many 
admissible weak solutions $(\rho,u)$ with the following additional mixing property: For any $r>0$, $0<t_0<T$ and $x_0\in\R^2$ such that $B:=B_r(x_0,t_0)\subset \Omega_{mix}$, both sets $\{(x,t)\in B:\,\rho(x,t)=\pm 1\}$ have strictly positive Lebesgue measure. 

Furthermore, there exists a sequence of such admissible weak solutions $(\rho_k,u_k)$ such that $\rho_k\overset{*}\rightharpoonup \bar{\rho}$ as $k\to\infty$. 
\end{theorem}

Thus, the crux of the matter is the construction of an admissible subsolution. In \cite{sze:ipm} the $x_1$-invariance of the initial curve $z_0 \equiv 0$ simplifies the construction of a subsolution. Indeed, assuming that $(\rho,u,m)$ is a function of $(x_2,t)$ only, the equation $\partial_t \rho + \div m=0$ together with maximizing the constraint \eqref{str} leads to Burger's equation $\partial_t \rho + c \partial_{x_2} \frac{1}{2}\rho^2 = 0$ with $0\leq c<2$. This equation admits a continuous rarefaction wave solution for the density
\begin{equation}\label{e:flatrho}
\rho(x,t) = \begin{cases} 1 & x_2 > c t,\\ \frac{x_2}{c t} & |x_2| < c t , \\ -1 & x_2 < -c t \end{cases}
\end{equation}
in the unstable case, whereas in the stable case we merely obtain the stationary shock-wave
\begin{equation*}
\rho(x,t) = \begin{cases} -1 & x_2 > 0, \\ 1 & x_2 < 0. \end{cases}
\end{equation*}

Therefore $c$ can be thought of as a weak notion of \emph{mixing speed}, in the following sense: the general structure of weak solutions corresponding to this subsolution will be that there are three time-dependent regions: $\Omega^+(t)$, $\Omega^-(t)$ and $\Omega_{mix}(t)$, given by 
\begin{equation}\label{omega}
\begin{split}
\Omega^+(t) &= \{x \in \R^2 | x_2 > z(x_1,t) + ct\},\\
\Omega_{mix}(t) &= \{x \in \R^2 |z(x_1,t) - ct < x_2 < z(x_1,t) + ct\},\\
\Omega^-(t) &= \{x \in \R^2 | x_2 < z(x_1,t) - ct\}.\\
\end{split}
\end{equation}
for some curve $z(\cdot,t)$, with $\Omega_{mix}(t)$ expanding with speed $c$. The three open sets in Definition \ref{d:subsol} are then
$$
\Omega^{\pm}=\bigcup_{t>0}\Omega^{\pm}(t),\quad \Omega_{mix}=\bigcup_{t>0}\Omega_{mix}(t).
$$
In $\Omega^{\pm}$ the density is given by the constant value $\rho^{\pm}=\pm 1$, and in the \emph{mixing zone} $\Omega_{mix}$ the two fluids are completely mixed - see Section 4 in \cite{sze:ipm} and Sections 2-3 in \cite{ccf:ipm}. In the above $x_1$-invariant setting from \cite{sze:ipm} the curve $z$ is simply stationary, i.e. $z(s,t)=z_0(s)\equiv 0$. Furthermore it was shown in \cite{sze:ipm} that for $x_1$-invariant subsolutions $c=2$ is the maximal possible speed, and it was conjectured, based on similarities with the Lagrangian relaxation framework of Otto in \cite{otto:ipm,otto:ema} that the maximal mixing speed could be used as a selection criterion.

In \cite{ccf:ipm} the construction from \cite{sze:ipm} was generalized to non-flat initial curves $z_0$ whilst retaining the structure \eqref{e:flatrho} in the mixing zone $\Omega_{mix}$. More precisely, the density is chosen as a linear interpolation between $\rho^+ = 1$ and $\rho^- = -1$. In this case, however, $z=z(s,t)$ has to solve a rather complicated evolution equation in time, which arises as a spatial average of the original Muskat evolution kernel. We wish to emphasize that the evolution equation obtained in \cite{ccf:ipm} is not necessarily a \emph{canonical} choice, but rather arises from the specific ansatz used for $\rho$ - indeed, we show below that simpler choices for the profile $\rho$ reduce the existence of a subsolution to differential inequalities which can be solved by prescribing velocity and acceleration of the interfaces at time $t=0$. Indeed, given $z:\R \times [0,T] \rightarrow \R$ define $\Omega^{\pm}(t)$ and $\Omega_{mix}(t)$ as in \eqref{omega} and set
\begin{align}\label{e:dens}
\begin{split}
\rho(x,t) = \begin{cases} \rho^+ & x \in \Omega^+(t),\\ 0 & x \in \Omega_{mix}(t) \\ \rho^- & x \in \Omega^-(t), \end{cases}
\end{split}
\end{align}
where $\rho^+=1, \rho^-=-1$ in the unstable case, and $\rho^+=-1, \rho^-=1$ in the stable case.
This definition of $\rho$ already determines the velocity $u$ by kinematic part of \eqref{IPMsub}, namely the Biot-Savart rule (see Section \ref{s:u} below)
\begin{equation}\label{e:BS}
\begin{split}
\div u&=0,\\
\curl u&=-\partial_{x_1}\rho.
\end{split}
\end{equation}
Note that $\rho$ is piecewise constant, with jump discontinuities across two interfaces 
\begin{equation}\label{e:Gammapm}
\Gamma^{\pm}(t)=\left\{ (s,z(s,t)\pm ct):\,s\in\R\right\}.
\end{equation}
It is well known \cite{cg:condynipm} that, provided the interfaces are sufficiently regular, the solution $u$ to \eqref{e:BS} is globally bounded, smooth in $\R^2\setminus(\Gamma^+\cup\Gamma^-)$ with well-defined traces on $\Gamma^{\pm}$, and the normal component is continuous across the interfaces. In particular it follows that the normal velocity component
\begin{equation}\label{e:unu}
u_{\nu}(x,t):=u(x,t)\cdot \begin{pmatrix}-\partial_{x_1}z(x_1,t)\\ 1\end{pmatrix}
\end{equation}
is a globally defined bounded and continuous function. In particular we set $u_{\nu}^\pm=u_{\nu}\vert_{\Gamma^{\pm}}$, i.e.
\begin{equation*}
u_{\nu}^{\pm}(s,t):=u_{\nu}(s,z(s,t)\pm ct,t).	
\end{equation*}

Our main result in this section is as follows:
\begin{theorem}\label{Subs}
Suppose that $z(s,t)=\beta s+\bar{z}(s,t)$ with $\bar{z}\in C^1([0,T];C^{1,\alpha}_*(\R))$ satisfies
\begin{align}
\lim_{t\to 0}\left \|\partial_t z(\cdot,t) - u_{\nu}^\pm(\cdot,t)\right\|_{L^{\infty}} &=  0\,, \label{z1}\\
\lim_{t\to 0}\frac{1}{t}\left\|2\partial_t z(\cdot,t) - u_\nu^+(\cdot,t) - u_{\nu}^-(\cdot,t)\right\|_{L^1} &=0\,. \label{z2} 
\end{align}
In the stable case assume in addition that $\|\bar{z}(\cdot,0)\|_{L^{\infty}}<|\beta|$. 
Then, there exists $T_*\in (0,T]$ such that there exists an admissible subsolution for \eqref{IPM1}-\eqref{IPM4} on $[0,T_*)$ with initial datum $\rho_0$ given by \eqref{init} with $z_0=z\vert_{t=0}$.  Furthermore, the density of the subsolution can be chosen to satisfy \eqref{omega}-\eqref{e:dens} with any $0<c<c_{max}$, where 
$$
c_{max}=\begin{cases}1&\textrm{ in the unstable case;}\\
\tfrac{1}{2}\tfrac{|\beta|(|\beta|-\|\partial_s\bar{z}_0\|_{L^{\infty}})}{1+|\beta|\|\partial_s\bar{z}_0\|_{L^{\infty}}}&\textrm{  in the stable case.}\end{cases}
$$
\end{theorem}

\begin{remark}
We note that the time of existence $T_*>0$ depends on $c$ and in particular $T_*\to 0$ as $c\to c_{max}$.
\end{remark}

\begin{proof}
Given $z=z(s,t)$, $c>0$ and $\rho(x,t)$ (defined by \eqref{omega}-\eqref{e:dens}), the velocity $u$ is determined by \eqref{e:BS}. Therefore it remains to define $m$ so that \eqref{IPMsub}-\eqref{str} are satisfied in $(0,T_*)\times \R^2$, with \eqref{str} a strict inequality in $\Omega_{mix}$. Set
\begin{align*}
m = \rho u - (1-\rho^2) (\gamma + \tfrac{1}{2}e_2)
\end{align*}
for some $\gamma=\gamma(x,t)$, with $\gamma\equiv 0$ in $\Omega^{\pm}$. Then \eqref{str} amounts to the condition
$$
|\gamma| < \frac{1}{2}\quad\textrm{ in }\Omega_{mix},
$$
whereas \eqref{IPMsub} is equivalent to $\div\gamma=0$ in $\Omega_{mix}$ together with two jump conditions
\begin{equation}\label{e:jump}
[\rho]_{\Gamma^{\pm}}(\partial_tz\pm c)+[m]_{\Gamma^{\pm}}\cdot \begin{pmatrix}\partial_{x_1}z(x_1,t)\\ -1\end{pmatrix}=0,
\end{equation}
where $[\cdot]_{\Gamma^\pm}$ denotes the jump on $\Gamma^{\pm}$.
Noting that $u_{\nu}$ in \eqref{e:unu} is globally well-defined and continuous, the jump conditions become
\begin{equation*}
[\rho]_{\Gamma^\pm}(\partial_t z - u_\nu^{\pm}\pm c)\mp\tfrac12\mp \gamma^{\pm}_{\nu}=0,
\end{equation*}
where 
\begin{equation*}
\gamma_{\nu}=\gamma\cdot \begin{pmatrix}-\partial_{x_1}z\\ 1\end{pmatrix}, 
\end{equation*}
and $\gamma_{\nu}^{\pm}$ denotes the one-sided limit $\lim_{\Omega_{mix}\ni x'\to x}\gamma_{\nu}(x)$ for $x\in \Gamma^{\pm}$.
Choosing $\gamma = \nabla^{\perp} g=\begin{pmatrix} -\partial_{x_2}g\\ \partial_{x_1}g\end{pmatrix}$ for some function $g\in C^1(\overline{\Omega_{mix} })$ and noting that
\begin{equation*}
[\rho]_{\Gamma^\pm}=\begin{cases}1&\textrm{ unstable case,}\\ -1&\textrm{ stable case,}\end{cases}
\end{equation*}
the conditions \eqref{IPMsub}-\eqref{str} reduce to
\begin{align}
|\nabla g| &< \frac{1}{2} \quad \textrm{in } \Omega_{mix} \label{g1}\\ 
\partial_\tau g&=(c - \tfrac{1}{2})\pm(\partial_t z - u_\nu^{\pm}) \quad \text{on }\Gamma^{\pm}\quad\textrm{(unstable case)}\label{g2}\\
\partial_\tau g&=-(c+\tfrac12) \mp (\partial_t z - u_\nu^{\pm} )\quad \text{on }\Gamma^{\pm}\quad\textrm{(stable case)}\label{g3}
\end{align}
where 
$$
\partial_\tau g(x,t)=\gamma_{\nu}(x,t)=\partial_{x_1}g(x,t)+\partial_{x_2}g(x,t)\partial_{x_1}z(x_1,t)
$$
is the tangential derivative of $g$ along curves defined by $z$.
We treat the unstable and stable cases separately.

\bigskip

\noindent{\bf Unstable case. } 
For $s\in\R$, $t\in (0,T)$ and $\lambda\in [-ct,ct]$ define
$$
\hat{g}(s,\lambda,t):=g(s,z(s,t)+\lambda,t)
$$
and observe that 
$$
(\partial_\tau g)(s,z(s,t)+\lambda,t)=\frac{\partial}{\partial s}\left(\hat{g}(s,\lambda,t)\right).
$$
In order to satisfy \eqref{g2} we then set
$$
\hat{g}(s,\pm ct,t):=\int_{0}^s c-\tfrac{1}{2}\pm(\partial_t z -u_\nu^\pm)\,ds',
$$
and, more generally, for $\lambda\in [-ct,ct]$ 
\begin{align*}
\hat{g}(s,\lambda,t)&:=\frac{ct+\lambda}{2ct}\hat{g}(s,ct,t)+\frac{ct-\lambda}{2ct}\hat{g}(s,-ct,t)\\
&= s ( c-\tfrac{1}{2} ) + \frac{\lambda+ct}{2ct} \left( \int_0^{s} \partial_t z - u_\nu^+ ds' \right) + \frac{\lambda-ct}{2ct} \left(\int_0^{s} \partial_t z - u_\nu^- ds' \right).
\end{align*}
Then
\begin{align*}
\partial_{\lambda}\hat{g}(s,\lambda,t)&=\frac{1}{2ct}\int_0^s\left(2\partial_tz-u_\nu^+-u_{\nu}^-\right)\,ds'\,,\\
\partial_s\hat{g}(s,\lambda,t)&=(c-\tfrac{1}{2})+\frac{\lambda+ct}{2ct} \left(\partial_t z - u_\nu^+ \right) + \frac{\lambda-ct}{2ct} \left( \partial_t z - u_\nu^-  \right).
\end{align*}
Noting that $\partial_{\lambda}\hat{g}(s,\lambda,t)=\partial_{x_2}g(s,z(s,t)+\lambda,t)$ and $\partial_s\hat{g}(s,\lambda,t)=\partial_{x_1}g(s,z(s,t)+\lambda,t)+\partial_{x_2}g(s,z(s,t)+\lambda,t)\partial_{x_1}z(s,t)$, from the assumptions \eqref{z1}-\eqref{z2} we deduce
\begin{align*}
\|\partial_{x_2}g(\cdot,t)\|_{L^{\infty}}\to 0\quad\textrm{ and }\|\partial_{x_1}g(\cdot,t)-(c-\tfrac12)\|_{L^{\infty}}\to 0
\end{align*}
as $t\to 0$. Therefore, for any $0<c<1$ we deduce that  
\begin{align*}
|\nabla g| < \frac{1}{2}\quad\textrm{ for sufficiently small }t>0.
\end{align*}
This concludes the proof in the unstable case.

\smallskip

\noindent{\bf Stable case. }
Define the one-parameter family of diffeomorphisms 
$$
\Phi_t(s,\lambda)=\begin{pmatrix}s-\tfrac{\beta}{1+\beta^2}\lambda\\ z\left(s-\tfrac{\beta}{1+\beta^2}\lambda,t\right)+\lambda\end{pmatrix}\,
$$
with inverse map
$$
\Psi_t(x_1,x_2)=\begin{pmatrix}x_1+\tfrac{\beta}{1+\beta^2}(x_2-z(x_1,t))\\ x_2-z(x_1,t)\end{pmatrix}.
$$
Since $\det D\Phi_t=1$, it follows that $\Phi_t$ is a global $C^1$-diffeomorphism of $\R^2$. Moreover,  
since $\Phi_t(s,\lambda)\in \{x_2=z(x_1)+\lambda\}$ for any $s\in \R$, it follows that $\Phi_t$  maps $\R\times [-ct,ct]$ onto $\overline{\Omega_{mix}}$. Set
$$
\hat{g}(s,\lambda,t)=g(\Phi_t(x),t).
$$
and observe that 
$$
(\partial_\tau g)\left(s-\tfrac{\beta}{1+\beta^2}\lambda,z\left(s-\tfrac{\beta}{1+\beta^2}\lambda,t\right)+\lambda,t\right)=\frac{\partial}{\partial s}\left(\hat{g}(s,\lambda,t)\right).
$$
In order to satisfy \eqref{g3} we set
\begin{align*}
\hat{g}(s,\pm ct,t):=\int_{0}^{s\mp\tfrac{\beta}{1+\beta^2}ct} -(c+\tfrac{1}{2})\mp(\partial_t z -u_\nu^\pm)(s')\,ds'
\end{align*}
and, more generally, for $\lambda\in [-ct,ct]$ 
\begin{align*}
\hat{g}(s,\lambda,t):=\frac{ct+\lambda}{2ct}\hat{g}(s,ct,t)+\frac{ct-\lambda}{2ct}\hat{g}(s,-ct,t).
\end{align*}
Then
\begin{align*}
\partial_{\lambda}\hat{g}&(s,\lambda,t)=\frac{1}{2ct}(\hat{g}(s,ct,t)-\hat{g}(s,ct,t))=-\frac{1}{2ct}\int_0^s\left(2\partial_tz-u_\nu^+-u_{\nu}^-\right)\,ds'+\\
&+\frac{1}{2ct}\int_{s-\tfrac{\beta}{1+\beta^2}ct}^s(\partial_t z -u_\nu^+)(s')\,ds'+\frac{1}{2ct}\int_s^{s+\tfrac{\beta}{1+\beta^2}ct}(\partial_t z -u_\nu^-)(s')\,ds'\,,\\
\partial_s\hat{g}&(s,\lambda,t)=-(c+\tfrac{1}{2})+\\
&-\frac{\lambda+ct}{2ct} \left(\partial_t z - u_\nu^+ \right)(s-\tfrac{\beta}{1+\beta^2}ct) - \frac{\lambda-ct}{2ct} \left( \partial_t z - u_\nu^-  \right)(s+\tfrac{\beta}{1+\beta^2}ct).
\end{align*}
From the assumptions \eqref{z1}-\eqref{z2} it follows that
\begin{align*}
\|\partial_{\lambda}\hat{g}(\cdot,t)\|_{L^{\infty}}\to 0\quad\textrm{ and }\|\partial_{s}\hat{g}(\cdot,t)+(c+\tfrac12)\|_{L^{\infty}}\to 0
\end{align*}
as $t\to 0$. 
Furthermore, using that $\nabla g(x)=D\Psi^T(x)\nabla \hat{g}(\Psi(x))$, we can estimate
$$
|\nabla g|\leq \frac{1+\beta|\partial_s\bar{z}|}{1+\beta^2}|\partial_s\hat{g}|+(\sqrt{1+\beta^2}+|\partial_s\bar{z}|)|\partial_{\lambda}\hat{g}|.
$$
Therefore we obtain
\begin{align*}
|\nabla g|\leq \frac{1+|\beta|\|\partial_s\bar{z}\|_{L^{\infty}}}{1+\beta^2}(c+\tfrac12)+o(1)
\end{align*}
as $t\to 0$. Hence, provided 
\begin{equation}\label{e:stablesmallness}
\|\partial_s\bar{z}_0\|_{L^{\infty}}<|\beta|,
\end{equation}
for any $c<\tfrac{1}{2}\tfrac{|\beta|(|\beta|-\|\partial_s\bar{z}_0\|_{L^{\infty}})}{1+|\beta|\|\partial_s\bar{z}_0\|_{L^{\infty}}}$ we have $|\nabla g|<1/2$  for sufficiently small $t>0$. This concludes the proof in the stable case.

\end{proof}

%%%%%%%%%%%%%%%%%%%%%%%%%%%%%%%%%%%%%%%%%%%%%%%%%%%%%%%%%%%%%%%%%%%%%%%%%%%%

\section{The velocity $u$}\label{s:u}

In this section we derive a concrete representation formula for the velocity $u$ and for the normal velocity component $u_\nu$ defined in \eqref{e:unu}, where $u$ is the solution of the system \eqref{e:BS}. It is well-known  \cite{mb:vorticity} that for sufficiently smooth $\rho$ the solution $v$ of 
\begin{align*}
\begin{cases}
\div v &= 0\\
\curl  v &= -\partial_{x_1} \rho
\end{cases} \quad \text{in }\R^2
\end{align*}
can be written using the Biot-Savart kernel as
\begin{align}\label{e:BSop}
v(x):= BS(-\partial_{x_1}\rho) := \frac{1}{2\pi} \int_{\R^2} \frac{(x-y)^{\perp}}{|x-y|^2} (-\partial_{x_1}\rho)(y) dy.
\end{align}
If the density $\rho$ is piecewise constant, with a jump across a sufficiently smooth interface $\Gamma$, the expression for $v(x)$ for $x\notin\Gamma$ can be derived by formally writing $\partial_{x_1} \rho$ as a delta distribution supported on $\Gamma$ \cite{cg:condynipm}. More precisely, in \cite{cg:condynipm}Ê(see Section 2 therein) the following expression was derived for the normal velocity component $v_\nu$ under the assumption that the interface is given by a graph $\Gamma=\{(s,z(s)):\,s\in \R\}$ with $z\in C^{1,\alpha}(\R)$ and 
$$
\rho(x)=\begin{cases} \rho^+& x_2>z(x_1),\\ \rho^-& x_2<z(x_1).\end{cases}
$$
For any $x=(x_1,x_2)\in\R^2$ we have
\begin{equation}\label{e:unu1}
\begin{split}
v_\nu(x)&:=v(x)\cdot\begin{pmatrix}-\partial_1z(x_1)\\1\end{pmatrix}\\
&=\frac{\rho^+-\rho^-}{2\pi}\textrm{PV}\int_{\R}\frac{\left(\partial_1z(x_1-\xi)-\partial_1z(x_1)\right)\xi}{\xi^2+(z(x_1-\xi)-x_2)^2}\,d\xi\,,
\end{split}
\end{equation}
where the principal value refers to the limit $\lim_{R\to\infty}\int_{-R}^R$. 

For the convenience of the reader we recall the argument leading up to formula \eqref{e:unu1}. First of all, by writing $\partial_{x_1} \rho$ as a delta distribution supported on $\Gamma$, from \eqref{e:BSop} one obtains
$$
v(x)=\frac{\rho^+-\rho^-}{2\pi}\textrm{PV}\int_{\R}\begin{pmatrix}z(\xi)-x_2\\ x_1-\xi\end{pmatrix}\frac{1}{(x_1-\xi)^2+(z(\xi)-x_2)^2}\partial_sz(\xi)\,d\xi
$$
for all $x\notin \Gamma$. Then, by using that
$$
\textrm{PV}\int_{\R}\partial_{\xi}\log\left((x_1-\xi)^2+(x_2-z(\xi))^2\right)\,d\xi=0,
$$
we deduce
$$
v(x)=\frac{\rho^+-\rho^-}{2\pi}\textrm{PV}\int_{\R}\begin{pmatrix}1\\ \partial_sz(\xi)\end{pmatrix}\frac{x_1-\xi}{(x_1-\xi)^2+(z(\xi)-x_2)^2}\,d\xi
$$
from which \eqref{e:unu1} follows. 

%Furthermore, it follows easily by linearity that, in the case of $N$ curves $\Gamma_i$ given as graphs of $N$ functions $z_i$ with $z_i(s) > z_{i+1}(s)$ for all $s \in \R, i = 1, \ldots, N-1$ with
%$$
%\rho(x) = \begin{cases} \rho_1 & x_2 > z_1(x_1),\\ \rho_2 & z_1(x_1) > x_2 > z_2(x_1),\\ \vdots & \quad \vdots\\ \rho_{N+1} & x_2 < z_N(x_1),  \end{cases}
%$$
%for $\rho_1, \rho_2, \ldots, \rho_{N+1} \in \R$,
%one easily obtains the analogous expression
%\begin{equation}\label{e:manyjumps}
%v_\nu(x) = \sum\limits_{i=1}^N \frac{\rho_i-\rho_{i+1}}{2\pi} \text{ PV}\int_{\R} \frac{(\partial_1z_i(x_1-\xi)-\partial_1z_i(x_1))\xi}{\xi^2 + (z_i(x_1-\xi)-x_2)^2}\, d\xi.
%\end{equation}

Then, for the density $\rho=\rho(x,t)$ defined in \eqref{e:dens} with interfaces $\Gamma^\pm$ in \eqref{e:Gammapm}, the normal velocity component \eqref{e:unu} on $\Gamma^\pm$ is given by the expression
\begin{equation}\label{e:u_nu}
u_{\nu}^{\pm}(s,t)=\frac{\rho^+-\rho^-}{4\pi}\textrm{PV}\int_{\R}\frac{\partial_sz(s-\xi,t)-\partial_sz(s,t)}{\xi}\Phi_{\pm}(\xi,s,t)\,d\xi,
\end{equation}
where 
$$
\Phi_{\pm}(\xi,s,t)=\frac{\xi^2}{\xi^2+(z(s-\xi,t)-z(s,t))^2}+\frac{\xi^2}{\xi^2+(z(s-\xi,t)-z(s,t)\mp 2ct)^2}\,.
$$
Motivated by this expression, consider the following $\Phi$-weighted variant of the Hilbert transform:
\begin{equation}\label{e:TPhi}
T_{\Phi}(f)(s):=\frac{1}{2\pi}\textrm{PV}\int_{\R}\frac{\partial_sf(s-\xi)-\partial_sf(s)}{\xi}\Phi(\xi,s)\,d\xi,
\end{equation}
for some bounded weight-function $\Phi=\Phi(\xi,s)$. For the weight we use the following norms: first of all we 
assume that $\Phi^{\infty}(s):=\lim_{|\xi|\to\infty}\Phi(\xi,s)$ exists, $\Phi(\cdot,s)\in C^1(\R\setminus\{0\})$, and set
\begin{equation}\label{e:farfieldPhi}
\begin{split}
\bar{\Phi}&=\xi\left(\Phi-\Phi^{\infty}\right),\quad \Phi^{\infty}=\lim_{|\xi|\to\infty}\Phi(\xi,s),\\
\tilde\Phi&=\xi^2\partial_{\xi}\left(\frac{1}{\xi}\Phi\right)=\xi\partial_{\xi}\Phi-\Phi\,.
\end{split}
\end{equation}
We introduce the norms
\begin{align*}
\vvvert \Phi\vvvert_0&:=\sup_{s\in\R,|\xi|\leq 1}|\Phi(\xi,s)|+\sup_{s\in\R,|\xi|>1}(|\bar{\Phi}(\xi,s)|+|\tilde\Phi(\xi,s)|),\\
\vvvert\Phi\vvvert_{k,\alpha}&:=\max_{j\leq k}\vvvert \partial_s^j\Phi\vvvert_0+[\partial_s^k\Phi]_\alpha+\sup_{|\xi|>1}([\partial_s^k\bar{\Phi}(\xi,\cdot)]_\alpha+[\partial_s^k\tilde\Phi(\xi,\cdot)]_\alpha),
\end{align*}
where we use the convention that $\|\Phi(\xi,\cdot)\|$ denotes a norm in the second argument only and $\|\Phi\|$ denotes a norm joint in both variables. In particular the H\"older-continuity of $\partial_s^k\Phi$ in both variables $\xi,s$ is required in the norm $\vvvert\Phi\vvvert_{k,\alpha}$. Accordingly, we define the spaces
\begin{align*}
\mathcal{W}^0&=\{\Phi\in L^{\infty}(\R^2):\,\Phi^\infty\textrm{ and }\partial_\xi\Phi\textrm{ exist, with }\vvvert \Phi\vvvert_0<\infty\},\\
\mathcal{W}^{k,\alpha}&=\{\Phi\in \mathcal{W}^0:\,\vvvert \Phi\vvvert_{k,\alpha}<\infty\}
\end{align*}
We have the following version of the classical estimate on the Hilbert transform $T_1=\mathcal{H}\nabla$ on H\"older-spaces:

\begin{theorem}\label{t:TPhi}
For any $\alpha>0$, $f\in C^{1,\alpha}_*(\R)$ and $\Phi\in \mathcal{W}^0$ we have
\begin{equation}\label{e:TPhi1}
\sup_{s}(1+|s|^{1+\alpha})|T_{\Phi}(f)(s)|\leq C \vvvert\Phi\vvvert_0\|f\|_{1,\alpha}^*\,.
\end{equation}
Moreover, for any $k\in\N$, $f\in C_*^{k+1,\alpha}(\R)$ and $\Phi\in \mathcal{W}^{k,\alpha}$ 
\begin{equation}\label{e:TPhi2}
\|T_{\Phi}(f)\|_{k,\alpha}^*\leq C\vvvert\Phi\vvvert_{k,\alpha}\|f\|_{k+1,\alpha}^*.
\end{equation}
where the constant depends only on $k$ and $\alpha$. 
\end{theorem}

\begin{proof}
We start the proof by rewriting the principal value integral in \eqref{e:TPhi} as a sum of absolutely convergent terms. To this end we split the integral
$\int_{\R}\,d\xi=\int_{|\xi|<1}\,d\xi+\int_{|\xi|>1}\,d\xi$ and integrate by parts in the second term. We obtain
\begin{align*}
T_{\Phi}(f)(s)=&\frac{1}{2\pi}\int_{|\xi|<1}\frac{\partial_sf(s-\xi)-\partial_sf(s)}{\xi}\Phi(\xi,s)\,d\xi\\
&-\frac{1}{2\pi}\partial_sf(s)\textrm{PV}\int_{|\xi|>1}\frac{1}{\xi}\Phi(\xi,s)\,d\xi\\
&+\frac{1}{2\pi}\left(f(s-1)\Phi(1,s)+f(s+1)\Phi(-1,s)\right)\\
&+\frac{1}{2\pi}\int_{|\xi|>1}f(s-\xi)\partial_{\xi}\left(\frac{1}{\xi}\Phi(\xi,s)\right)\,d\xi\\
=&\frac{1}{2\pi}\int_{|\xi|<1}\frac{\partial_sf(s-\xi)-\partial_sf(s)}{\xi}\Phi(\xi,s)\,d\xi\tag{$I_1$}\\
&-\frac{1}{2\pi}\partial_sf(s)\int_{|\xi|>1}\frac{1}{\xi^2}\bar{\Phi}(\xi,s)\,d\xi\tag{$I_2$}\\
&+\frac{1}{2\pi}\left(f(s-1)\Phi(1,s)+f(s+1)\Phi(-1,s)\right)\tag{$I_3$}\\
&+\frac{1}{2\pi}\int_{|\xi|>1}\frac{f(s-\xi)}{\xi^2}\tilde\Phi(\xi,s)\,d\xi,\tag{$I_4$}
\end{align*}
where $\bar\Phi$ and $\tilde\Phi$ are related to $\Phi$ as in \eqref{e:farfieldPhi}.

It is easy to see that 
$$
\sup_s(1+|s|^{1+\alpha})|I_i(s)|\leq C\vvvert\Phi\vvvert_0\|f\|_{1,\alpha}^*
$$
for $i=1,2,3$. For the term $I_4$ observe that if $|\xi|<\tfrac12|s|$ or $|\xi|>\tfrac32|s|$, then $|s-\xi|>\tfrac12|s|$. Therefore we have for any $|s|>2$
\begin{align*}
|I_4|&\leq \int_{1<|\xi|<\tfrac12|s|\textrm{ or }|\xi|>\tfrac32|s|}\frac{|f(s-\xi)|}{\xi^2}|\tilde\Phi(\xi,s)|\,d\xi+\int_{\tfrac12|s|<|\xi|<\tfrac32|s|}\frac{|f(s-\xi)|}{\xi^2}|\tilde\Phi(\xi,s)|\,d\xi\\
&\leq C|s|^{-(1+\alpha)}\|f\|_{1,\alpha}^*\int_{|\xi|>1}\frac{|\tilde\Phi(\xi,s)|}{\xi^2}\,d\xi+\frac{4\|f\|_{1,\alpha}^*}{s^2}\int_{\tfrac12|s|<|\tau|<\tfrac52|s|}\frac{|\tilde\Phi(s-\tau,s)|}{1+|\tau|^{1+\alpha}}\,d\tau\\
&\leq C  |s|^{-(1+\alpha)}\sup_{s\in\R,|\xi|>1}|\tilde\Phi(\xi,s)|\|f\|_{1,\alpha}^*\,.
\end{align*}
This concludes the proof of \eqref{e:TPhi1}. 

The H\"older continuity of $I_2$, $I_3$ and $I_4$ is easily handled analogously and leads to the estimates
\begin{align*}
[I_2]_{\alpha}^*&\leq C\|\partial_sf\|_{\alpha}^*\sup_{|\xi|>1}\|\bar{\Phi}(\xi,\cdot)\|_\alpha\,,\\
[I_3]_{\alpha}^*&\leq C\|f\|_{\alpha}^*\sup_{\xi}\|\Phi(\xi,\cdot)\|_{\alpha}\,,\\
[I_4]_{\alpha}^*&\leq C\|f\|_{\alpha}^*\sup_{|\xi|>1}\|\tilde{\Phi}(\xi,\cdot)\|_{\alpha}\,.
\end{align*}
Next, we consider $I_1=I_1(s)$ and write for simplicity $g(s)=\partial_sf(s)$. For $|\eta|<1/2$ let $\tilde s=s-\eta$ and write
\begin{align*}
I_1(\tilde s)&=\int_{|s-\eta-\xi|<1}\frac{g(\xi)-g(\tilde s)}{\tilde s-\xi}\Phi(\tilde s-\xi,\tilde s)\,d\xi\\
&=\int_{|s-\xi|<1}\frac{g(\xi)-g(\tilde s)}{\tilde s-\xi}\Phi(\tilde s-\xi,\tilde s)\,d\xi+I_{11},
\end{align*}
where $I_{11}$ is an integral over intervals of total length $\sim|\eta|$ on which $|\tilde s-\xi|>1/2$. Therefore
$$
|I_{11}|\leq C|\eta|(1+|s|^{1+\alpha})^{-1}\|f\|_{1,\alpha}^*\vvvert\Phi\vvvert_0.
$$
Next, we write, with $r=2|\eta|$
\begin{align*}
I_1(\tilde s)-I_1(s)-I_{11}&=\int_{|s-\xi|<1}\frac{g(\xi)-g(\tilde s)}{\tilde s-\xi}\Phi(\tilde s-\xi,\tilde s)-\frac{g(\xi)-g(s)}{s-\xi}\Phi(s-\xi,s)\,d\xi\\
=&\int_{|s-\xi|<r}\frac{g(\xi)-g(\tilde s)}{\tilde s-\xi}\Phi(\tilde s-\xi,\tilde s)-\frac{g(\xi)-g(s)}{s-\xi}\Phi(s-\xi,s)\,d\xi\tag{$I_{12}$}\\
&+\int_{r<|s-\xi|<1}\frac{g(s)-g(\tilde s)}{\tilde s-\xi}(\Phi(\tilde s-\xi,\tilde s)-\Phi(0,\tilde s))\,d\xi\tag{$I_{13}$}\\
&+\int_{r<|s-\xi|<1}(g(\xi)-g(s))(\frac{1}{\tilde s-\xi}-\frac{1}{s-\xi})\Phi(\tilde s-\xi,\tilde s),d\xi\tag{$I_{14}$}\\
&+\int_{r<|s-\xi|<1}\frac{g(\xi)-g(s)}{s-\xi}(\Phi(\tilde s-\xi,\tilde s)-\Phi(s-\xi,s))\,d\xi\tag{$I_{15}$}\\
&+\int_{r<|s-\xi|<1}\frac{g(s)-g(\tilde s)}{\tilde s-\xi}\Phi(0,\tilde s),d\xi.\tag{$I_{16}$}
\end{align*}
We can estimate each term as follows:
\begin{align*}
(1+|s|^{1+\alpha})|I_{12}|&\leq C[g]_{\alpha}^*\|\Phi\|_0\int_{|s-\xi|<r}|\tilde s-\xi|^{\alpha-1}+|s-\xi|^{\alpha-1}\,d\xi \\
&\leq C[g]_{\alpha}^*\|\Phi\|_0|\eta|^{\alpha}\\
(1+|s|^{1+\alpha})|I_{13}|&\leq C[g]_{\alpha}^*[\Phi]_{\alpha}|\eta|^{\alpha}\int_{|s-\xi|<1}|\tilde s-\xi|^{\alpha-1}\,d\xi\\
&\leq C[g]_{\alpha}^*[\Phi]_\alpha|\eta|^{\alpha}\\
(1+|s|^{1+\alpha})|I_{14}|&\leq C[g]_{\alpha}^*\|\Phi\|_0|\eta|\int_{r<|s-\xi|<1}|s-\xi|^{\alpha-1}|\tilde s-\xi|^{-1}\,d\xi\\
&\leq C[g]_{\alpha}^*\|\Phi\|_0|\eta|^{\alpha}\\
(1+|s|^{1+\alpha})|I_{15}|&\leq C[g]_{\alpha}^*[\Phi]_{\alpha}|\eta|^{\alpha}
\end{align*}
whereas, using that $\int_{r<|s-\xi|<1}\frac{1}{s-\xi}\,d\xi=0$
\begin{align*}
(1+|s|^{1+\alpha})|I_{16}(s)|&=|g(s)-g(\tilde s)||\Phi(0,\tilde s)|\left|\int_{r<|s-\xi|<1}\frac{1}{\tilde s-\xi}-\frac{1}{s-\xi}\,d\xi\right|\\
&\leq C[g]_{\alpha}^*\|\Phi\|_0|\eta|^{\alpha}\,.
\end{align*}
We conclude that 
$$
\|I_1\|_{\alpha}^*\leq C\|\Phi\|_{\alpha}\|\partial_sf\|_{\alpha}^*,
$$
and this finally proves \eqref{e:TPhi2} for $k=0$. For $k\geq 1$ the estimate follows from differentiating the terms $I_1(s),\dots,I_4(s)$ with respect to $s$ and applying the Leibniz rule.
\end{proof}

We close this section by showing that, under quite general conditions, $\Phi_{\pm}$ belongs to the weight-space $\mathcal{W}^0$:

\begin{lemma}\label{l:Tpm}
	Suppose that $z=z(s,t)=\beta s+\bar{z}(s,t)$ with $\bar{z}\in C([0,T];C^{1,\alpha}(\R))$ for some $0<\alpha<1$, $\beta\in\R$ and $T<\infty$. Then $\Phi_{\pm}\in \mathcal{W}^0$ with $\sup_{t\in[0,T]}\vvvert\Phi_{\pm}\vvvert_0<\infty$.   
\end{lemma}

\begin{proof}
First of all we easily see that
$$
\sup_{s,\xi,t}|\Phi_{\pm}(\xi,s,t)|\leq 2.
$$
For simplifying the notation, set
\begin{equation}\label{e:Zt}
Z_t(\xi,s)=\frac{z(s,t)-z(s-\xi,t)}{\xi}=\int_0^1\partial_sz(s-\tau\xi,t)\,d\tau,
\end{equation}
so that we may write for any $\xi\neq 0$
$$
\Phi_{\pm}=\frac{1}{1+Z_t^2}+\frac{1}{1+(Z_t\pm\tfrac{2ct}{\xi})^2}.
$$
Moreover, observe that $\sup_{\xi,s}|Z_t(\xi,s)|\leq \|\partial_sz\|_{L^{\infty}}$ and, 
since $Z_t=\beta+\frac{\bar{z}(s,t)-\bar{z}(s-\xi,t)}{\xi}$, $\lim_{|\xi|\to\infty}Z_t(\xi,s)=\beta$ uniformly in $s\in\R$. 
Therefore, with the notation from \eqref{e:farfieldPhi},
$$
\Phi_{\pm}^{\infty}=\frac{2}{1+\beta^2},
$$
and
$$
\overline{\Phi_{\pm}}=\xi\frac{\beta^2-Z_t^2}{(1+Z_t^2)(1+\beta^2)}+\xi\frac{\beta^2-(Z_t\pm\tfrac{2ct}{\xi})^2}{(1+(Z_t\pm\tfrac{2ct}{\xi})^2)(1+\beta^2)}.
$$
Since
$$
\sup_{s,\xi}|\xi(Z_t\pm\tfrac{2ct}{\xi}-\beta)|=\sup_{s,\xi}|\bar{z}(s,t)-\bar{z}(s-\xi,t)\pm 2ct|\leq 2\|\bar{z}\|_{L^{\infty}}+2ct,
$$
we deduce that 
$$
\sup_{s,|\xi|>1}|\overline{\Phi_{\pm}}|\leq C,
$$
with the constant $C$ depending on $\|\bar{z}\|_{L^{\infty}}$, $\beta$, $c$ and $T$.

Next, we calculate:
$$
\xi\partial_\xi\Phi_{\pm}=\frac{-2Z_t\xi\partial_\xi Z_t}{(1+Z_t^2)^2}+\frac{-2(Z_t\pm\tfrac{2ct}{\xi})(\xi\partial_\xi Z_t\mp\tfrac{2ct}{\xi})}{(1+(Z_t\pm\tfrac{2ct}{\xi})^2)^2}\,.
$$ 
Since 
$$
\xi\partial_\xi Z_t=\partial_sz(s-\xi,t)-\frac{z(s,t)-z(s-\xi,t)}{\xi}=\partial_s\bar{z}(s-\xi,t)-\frac{\bar{z}(s,t)-\bar{z}(s-\xi,t)}{\xi},
$$
we deduce that $\sup_{s,|\xi|>1}|\xi\partial_\xi\Phi_{\pm}|$ and hence $\sup_{s,|\xi|>1}|\tilde\Phi_{\pm}|$ is bounded uniformly in $t\in[0,T]$. This concludes the proof.

\end{proof}

%%%%%%%%%%%%%%%%%%%%%%%%%%%%%%%%%%%%%%%%%%%%%%%%%%%%%%%%%%%%%%%%%%%%%%%%%%%%

\section{Construction of the curve $z$}\label{s:z}

In this section we construct a function $z=z(s,t)$ satisfying the conditions of Theorem \ref{Subs}. 
In order to motivate the construction, observe that \eqref{z1}-\eqref{z2} suggest that it suffices to
specify $z$ up to order $t^2$. Therefore we start by formally calculating the expressions for the initial velocity $u_\nu^\pm\vert_{t=0}$ and initial symmetrized acceleration $\partial_t\vert_{t=0}\frac{u_\nu^++u_\nu^-}{2}$. 

Let $z=z(s,t)=\beta s+\bar{z}(s,t)$ with $\bar{z}\in C^2([0,T);C_*^{1,\alpha}(\R))$ for some $\beta\in\R$ and $\alpha\in(0,1)$. Using the expression \eqref{e:u_nu} and the notation introduced in \eqref{e:TPhi} we have
$$
u_{\nu}^\pm=\frac{\rho^+-\rho^-}{2}T_{\Phi_{\pm}}z,
$$
where $\rho^{\pm}=\pm 1$ in the unstable case and $\rho^{\pm}=\mp 1$ in the stable case. This difference in sign has no effect on the computations and on Theorem \ref{t:z} below, therefore we will from now on treat the unstable case without loss of generality.  
Hence, in particular
\begin{equation*}%\label{e:unu-0}
	u_{\nu}^\pm\big\vert_{t=0}=u_{\nu}^0:=T_{\Phi_{0}}\bar{z}_0,
\end{equation*}
where 
\begin{equation}\label{e:Phi0}
\Phi_0(\xi,s):=\frac{2\xi^2}{\xi^2+(z_0(s-\xi)-z_0(s))^2}
\end{equation}
and $z_0(s)=z(s,0)$. Observe that although $\Phi_{\pm}(\xi,s)\to\Phi_{0}(\xi,s)$ as $t\to 0$ for any $\xi\neq 0$, the limit is not uniform in $\xi$, therefore in particular $\Phi_{\pm}\nrightarrow \Phi_0$ in the norm of $\mathcal{W}^0$. Nevertheless we have

\begin{lemma}\label{l:continuity}
Assume that $z(s,t)=\beta s+\bar{z}(s,t)$ with $\bar{z}\in C^0([0,T);C_*^{1,\alpha}(\R))$ for some $\beta\in\R$ and $\alpha\in(0,1)$. Then for any $f\in C^{1,\alpha}_*(\R)$
$$
\lim_{t\to 0}\sup_{s\in\R}(1+|s|^{1+\alpha})\left|T_{\Phi_{\pm}}f(s)-T_{\Phi_0}f(s)\right|=0.
$$
\end{lemma}

\begin{proof}
In analogy with \eqref{e:Zt} we set
\begin{equation}\label{e:Z01}
Z_0(\xi,s)=\frac{z_0(s)-z_0(s-\xi)}{\xi},
\end{equation}
so that $\Phi_0=\frac{2}{1+Z_0^2}$.
In the following we consider without loss of generality $\Phi_+(t)-\Phi_0$.  

As pointed out above, the limit $\lim_{t\to 0}\Phi_+(t)$ is not uniform in $\xi$ because of the singularity at $\xi=0$. Therefore we need to modify the argument in the proof of  \eqref{e:TPhi1}. To this end recall the decomposition 
$$
T_{\Phi_{+}}f-T_{\Phi_0}f=T_{(\Phi_+(t)-\Phi_0)}f=I_1+I_2+I_3+I_4
$$
in the proof of Theorem \ref{t:TPhi}
and focus for the moment on the term
$$
I_1(s,t)=\frac{1}{2\pi}\int_{|\xi|<1}\frac{\partial_sf(s-\xi)-\partial_sf(s)}{\xi}(\Phi_+-\Phi_0)(\xi,s,t)\,d\xi.
$$
Using the definition of $\Phi_+$ and $\Phi_0$ we write for $\xi\neq 0$
\begin{align*}
\Phi_+&-\Phi_0=\frac{1}{1+Z_t^2}+\frac{1}{1+(Z_t+\frac{2ct}{\xi})^2}-\frac{2}{1+Z_0^2}\\
&=\frac{Z_0+Z_t}{(1+Z_0^2)(1+Z_t^2)}(Z_0-Z_t)+\frac{Z_0+Z_t+\frac{2ct}{\xi}}{(1+Z_0^2)(1+(Z_t+\tfrac{2ct}{\xi})^2)}(Z_0-Z_t-\tfrac{2ct}{\xi}).
\end{align*}
It is easy to see that $\sup_{\xi,s,t}|\Phi_+|\leq 2$. Moreover 
\begin{align*}
\sup_{\xi,s}|Z_t-Z_0|&\leq \sup_{s}|\partial_sz(s,t)-\partial_sz_0(s)|\overset{t\to 0}{\longrightarrow} 0,\\
\sup_{t^{1/2}<|\xi|}|\tfrac{2ct}{\xi}|&\leq 2ct^{1/2}\overset{t\to 0}{\longrightarrow} 0,
\end{align*}
hence
\begin{equation}\label{e:nonuniform1}
\lim_{t\to 0}\sup_{t^{1/2}<|\xi|,s\in\R}|\Phi_+(\xi,s,t)-\Phi_0(\xi,s)|=0.
\end{equation}
On the other hand, by splitting the integral $\int_{|\xi|<1}d\xi=\int_{|\xi|<t^{1/2}}d\xi+\int_{t^{1/2}<|\xi|<1}d\xi$ we can estimate
\begin{align*}
(1+|s|^{1+\alpha})|I_1(s,t)|\leq &C[\partial_sf]_{\alpha}^*(\sup_{\xi,s}|\Phi_+|+|\Phi_0|)\int_{|\xi|<t^{1/2}}|\xi|^{\alpha-1}\,d\xi+\\
&+ C[\partial_sf]_{\alpha}^*\bigl(\sup_{t^{1/2}<|\xi|,s}|\Phi_+-\Phi_0|\bigr)\int_{t^{1/2}<|\xi|<1}|\xi|^{\alpha-1}\,d\xi\\
\leq &C[\partial_sf]_{\alpha}^* \Bigl(t^{\alpha/2}+\sup_{t^{1/2}<|\xi|,s}|\Phi_+-\Phi_0|\Bigr).
\end{align*}
Hence
$$
\lim_{t\to 0}\sup_s(1+|s|^{1+\alpha})|I_1(s,t)|=0.
$$
Next, from \eqref{e:nonuniform1} we equally deduce
$$
\lim_{t\to 0}\sup_s(1+|s|^{1+\alpha})|I_3(s,t)|=0.
$$
Concerning $I_2$, note that for $|\xi|>1$
\begin{align*}
|\xi||\Phi_+-\Phi_0|&=|\xi|\left|\frac{Z_0^2-Z_t^2}{(1+Z_0^2)(1+Z_t^2)}+\frac{Z_0^2-(Z_t+\tfrac{2ct}{\xi})^2}{(1+Z_0^2)(1+(Z_t+\tfrac{2ct}{\xi})^2)}\right|\\
&\leq C\sup_{|\xi|>1,s}(|Z_0|+|Z_t|)\left(|\xi|\left|Z_0-Z_t\right|+|\xi|\left|Z_0-Z_t-\tfrac{2ct}{\xi}\right|\right)\\
&\leq C\|\partial_sz\|_{L^{\infty}}\left(|\xi|\left|Z_0-Z_t\right|+2ct\right).
\end{align*}
Since also
$$
\sup_{|\xi|>1,s\in\R}|\xi|\left|Z_0-Z_t\right|\leq 2\sup_{s\in\R}|z(s,t)-z_0(s)|\to 0\,\textrm{Êas }t\to 0,
$$
we deduce that $\lim_{t\to 0}\sup_{|\xi|>1,s\in R}|\overline{(\Phi_+-\Phi_0)}(\xi,s,t)|=0$ and consequently
$$
\lim_{t\to 0}\sup_s(1+|s|^{1+\alpha})|I_2(s,t)|=0.
$$
Finally, let us look at $I_4$, which requires bounding $\xi\partial_\xi(\Phi_+-\Phi_0)$. Observe that
$$
\xi\partial_\xi Z_t=\partial_sz(s-\xi,t)-\frac{z(s,t)-z(s-\xi,t)}{\xi},
$$
so that 
$$
\sup_{s,\xi}\left|\xi\partial_\xi Z_t\right|\leq 2\sup_{s}|\partial_sz(s,t)|\,\textrm{ and }\sup_{s,\xi}\left|\xi\partial_\xi(Z_t-Z_0)\right|\leq 2\sup_{s}|\partial_sz(s,t)-\partial_sz_0(s)|\,.
$$
It then follows by a simple calculation that $\lim_{t\to 0}\sup_{|\xi|>1,s\in \R}|\xi\partial_\xi(\Phi_+-\Phi_0)(\xi,s,t)|=0$ and hence
$$
\lim_{t\to 0}\sup_s(1+|s|^{1+\alpha})|I_4(s,t)|=0.
$$
This concludes the proof of the Lemma.
\end{proof}

Next, in order to evaluate $\left.\frac{\partial}{\partial t}\right\vert_{t=0}\frac{u_\nu^++u_\nu^-}{2}$, we 
first calculate
\begin{align*}
\frac{\Phi_++\Phi_-}{2}-\Phi_0&=\underbrace{\frac{Z_0^2-Z_t^2}{2(1+Z_0^2)}\left(\frac{2}{1+Z_t^2}+\frac{1}{1+(Z_t+\tfrac{2ct}{\xi})^2}+\frac{1}{1+(Z_t-\tfrac{2ct}{\xi})^2}\right)}_{\Delta_{reg}\Phi}\\
&+\underbrace{\frac{1}{1+Z_0^2}\left(\frac{-2ctZ_t\xi-2c^2t^2}{\xi^2+(Z_t\xi+2ct)^2}+\frac{2ctZ_t\xi-2c^2t^2}{\xi^2+(Z_t\xi-2ct)^2}\right)}_{\Delta_{sing}\Phi}\,.	
\end{align*}
Moreover, let
\begin{equation}\label{e:Phi1}
\Phi_1(\xi,s)=\left.\frac{\partial}{\partial t}\right\vert_{t=0}\frac{\Phi_++\Phi_-}{2}=\frac{-4Z_0Z_0'}{(1+Z_0^2)^2}\,.
\end{equation}
where $z_0'(s)=\partial_tz(s,0)$ and $Z_0'(\xi,s)=\frac{z_0'(s)-z_0'(s-\xi)}{\xi}$. For the regular part $\Delta_{reg}\Phi$ we have

\begin{lemma}\label{l:regular}
Assume $z=z(s,t)=\beta s+\bar{z}(s,t)$ with $\bar{z}\in C^1([0,T);C_*^{1,\alpha}(\R))$ for some $\beta\in\R$ and $\alpha\in(0,1)$. Then for any $f\in C^{1,\alpha}_*(\R)$
$$
\lim_{t\to 0}\sup_{s\in\R}(1+|s|^{1+\alpha})\left|\tfrac{1}{t}T_{\Delta_{reg}\Phi}f(s)-T_{\Phi_1}f(s)\right|=0.
$$
\end{lemma}

\begin{proof}
Observe that
$$
\frac{Z_t-Z_0}{t}-Z_0'=\int_0^1\frac{\partial_sz(s-\tau\xi,t)-\partial_sz_0(s-\tau\xi)}{t}-\partial_t\partial_sz_0(s-\tau\xi)\,d\tau,
$$
and by assumption 
$$
\sup_{s\in\R}\left|\frac{\partial_sz(s,t)-\partial_sz_0(s)}{t}-\partial_t\partial_sz_0(s)\right|\to 0\textrm{ as }t\to 0.
$$
Therefore
$$
\sup_{s,\xi}\left|\frac{Z_t-Z_0}{t}-Z_0'\right|\to 0\textrm{ as }t\to 0.
$$
Now write
\begin{align*}
\tfrac{1}{t}\Delta_{reg}\Phi-\Phi_1=\frac{Z_0^2-Z_t^2}{2t(1+Z_0^2)}\left(\Phi_++\Phi_--2\Phi_0\right)+\frac{2\Phi_0}{1+Z_0^2}\left(\frac{Z_0^2-Z_t^2}{2t}+Z_0Z_0'\right)\,.
\end{align*}
Since $\frac{Z_0^2-Z_t^2}{2t(1+Z_0^2)}$ is uniformly bounded in $s,\xi$ as $t\to 0$, the first summand can be dealt with exactly as in the proof of Lemma \ref{l:continuity}. On the other hand the second summand converges to zero uniformly in $s,\xi$ as $t\to 0$. This concludes the proof.
\end{proof}

Next we analyse the singular part $\Delta_{sing}\Phi$:

\begin{lemma}\label{l:singular}
Assume $z=z(s,t)=\beta s+\bar{z}(s,t)$ with $\bar{z}\in C^0([0,T);C_*^{1,\alpha}(\R))$ for some $\beta\in\R$ and $\alpha\in(0,1)$. Then, for any $f\in C^{2}(\R)$
\begin{equation}\label{e:TDeltasI1}
\lim_{t\to 0}\frac{1}{2\pi t}\int_{|\xi|<1}\frac{\partial_sf(s-\xi)-\partial_sf(s)}{\xi}\Delta_{sing}\Phi(\xi,s,t)\,d\xi=c\partial_s^2f(s)\sigma(s)
\end{equation}
for any $s\in \R$, 
where
$$
\sigma=\left.\frac{1-(\partial_sz)^2}{(1+(\partial_sz)^2)^2}\right|_{t=0}.
$$
Moreover, if in addition $f\in C^{2,\alpha}_*(\R)$, then 
\begin{equation}\label{e:TDeltas}
  \lim_{t\to 0}\sup_{s\in\R}(1+|s|^{1+\alpha})\left|\tfrac{1}{t}(T_{\Delta_{sing}\Phi}f)(s)-c\partial_s^2f(s)\sigma(s)\right|=0.
\end{equation}
\end{lemma}

\begin{proof}
We begin by performing the change of variables $\xi\mapsto \frac{\xi}{t}$ in the integral:
\begin{align*}
	\frac{1}{2\pi t}&\int_{|\xi|<1}\frac{\partial_sf(s-\xi)-\partial_sf(s)}{\xi}\Delta_{sing}\Phi(\xi,s,t)\,d\xi=\\
	&=\frac{1}{2\pi}\int_{|\xi|<\tfrac{1}{t}}\frac{\partial_sf(s-t\xi)-\partial_sf(s)}{t\xi}\Psi_t(\xi,s)\,d\xi,
\end{align*}
where
\begin{align*}
\Psi_t(\xi,s)=\Delta_{sing}\Phi(t\xi,s,t)=\frac{-4c^2(4c^2+(1-3Z_t^2)\xi^2)}{(1+Z_0^2)(\xi^2+(Z_t\xi+2c)^2)(\xi^2+(Z_t\xi-2c)^2)}.
\end{align*}
Since $\sup_{\xi,s,t}|Z_t|\leq\|\partial_sz\|_{L^{\infty}}$ and 
\begin{align*}
\xi^2+(Z_t\xi\pm 2c)^2&=\frac{4c^2}{1+Z_t^2}+(1+Z_t^2)\left(\xi\pm \frac{2Z_tc}{1+Z_t^2}\right)^2\\
&\geq \frac{4c^2}{1+\|\partial_sz\|^2_{L^\infty}},
\end{align*}
it follows that
\begin{equation}\label{e:DsPhi}
|\Psi_t(\xi,s)|\leq \frac{C}{1+|\xi|^2},
\end{equation}
where the constant depends only on $\|\partial_sz\|_{L^{\infty}}$ and on $c>0$. Furthermore, since
\begin{align*}
\lim_{t\to 0}Z_0(t\xi,s)=\lim_{t\to 0}Z_t(t\xi,s)\to \partial_sz_0(s)\textrm{ for all }\xi,s\in \R,	
\end{align*}
we deduce that for any $\xi,s\in\R$
\begin{align*}
\Psi_0(\xi,s)&=\lim_{t\to 0}\Psi_t(\xi,s)\\
&=\frac{1}{1+(\partial_sz_0)^2}\left(\frac{-2c\xi\partial_sz_0-2c^2}{\xi^2+(\xi\partial_sz_0+2c)^2}+\frac{2c\xi\partial_sz_0-2c^2}{\xi^2+(\xi\partial_sz_0-2c)^2}\right),
\end{align*}
where we write $z_0=z|_{t=0}$.
Finally, since $f\in C^{2}(\R)$, we have
$$
\sup_{|t\xi|<1}\left|\frac{\partial_sf(s-t\xi)-\partial_sf(s)}{t\xi}\right|\leq \sup_{|s-s'|<1}|\partial_s^2f(s')|,  
$$
so that, 
from the Lebesgue dominated convergence theorem we deduce that
\begin{align*}
\lim_{t\to 0}\frac{1}{2\pi}\int_{|\xi|<\tfrac{1}{t}}\frac{\partial_sf(s-t\xi)-\partial_sf(s)}{t\xi}\Psi_t(\xi,s)\,d\xi=-\frac{\partial_s^2f(s)}{2\pi}\int_{\R}\Psi_0(\xi,s)\,d\xi.
\end{align*}
Noting that the bound \eqref{e:DsPhi} applies also to $\Psi_0$ and $\partial_sz_0$ is independent of $\xi$, we may evaluate the integral $\int_{\R}\Psi_0\,d\xi$ using elementary methods. Indeed, we calculate for any constants $a\in\R$ and $c>0$ 
\begin{align*}
\mathcal{I}_{a,c}=&\frac{1}{2\pi}\int_{\R}\frac{-2ac\xi-2c^2}{(1+a^2)\xi^2+4ac\xi+4c^2}+\frac{2ac\xi-2c^2}{(1+a^2)\xi^2-4ac\xi+4c^2}\,d\xi\\
=\frac{1}{2\pi}&\int_{\R}\frac{ac}{1+a^2}\left(\frac{2(1+a^2)\xi-4ac}{(1+a^2)\xi^2-4ac\xi+4c^2}-\frac{2(1+a^2)\xi+4ac}{(1+a^2)\xi^2+4ac\xi+4c^2}\right)\,d\xi\\
&+\frac{1}{2\pi}\frac{2c^2(a^2-1)}{1+a^2}\int_{\R}\frac{1}{(1+a^2)\xi^2+4ac\xi+4c^2}+\frac{1}{(1+a^2)\xi^2-4ac\xi+4c^2}\,d\xi\\
=\frac{1}{2\pi}&\frac{ac}{1+a^2}\lim_{R\to \infty}\left[\log\frac{(1+a^2)\xi^2-4ac\xi+4c^2}{(1+a^2)\xi^2+4ac\xi+4c^2}\right]^R_{-R}\\
&+\frac{1}{2\pi}\frac{2c^2(a^2-1)}{(1+a^2)^2}\int_{\R}\frac{2}{\xi^2+\frac{4c^2}{(1+a^2)^2}}\,d\xi\\
=\frac{1}{2\pi}&\frac{2c(a^2-1)}{1+a^2}\int_{\R}\frac{1}{1+\xi^2}\,d\xi=-c\frac{1-a^2}{1+a^2}.
\end{align*}
Therefore
$$
\frac{1}{2\pi}\int_{\R}\Psi_0(\xi,s)\,d\xi=-c\frac{1-\partial_sz_0(s)^2}{(1+(\partial_sz_0(s))^2)^2},
$$
hence the proof of \eqref{e:TDeltasI1} follows.

In order to show \eqref{e:TDeltas} we again start with the decomposition
$$
\tfrac{1}{t}T_{\Delta_{sing}\Phi}f=I_1+I_2+I_3+I_4
$$
as in the proof of Theorem \ref{t:TPhi}. We claim that 
\begin{equation}\label{e:SingClaim1}
\sup_s(1+|s|^{1+\alpha})|I_1(s)-c\partial_s^2f(s)\sigma(s)|\to 0\,\textrm{ as }t\to 0,
\end{equation}
and for $k=2,3,4$
\begin{equation}\label{e:SingClaim2}
\sup_s(1+|s|^{1+\alpha})|I_k(s)|\to 0\,\textrm{ as }t\to 0.
\end{equation}
The proof of \eqref{e:SingClaim2} follows from the observation that $\tfrac{1}{t}\Delta_{sing}\Phi$ and $\xi\partial_{\xi}\tfrac{1}{t}\Delta_{sing}\Phi$ are bounded uniformly in $s\in\R$, $|\xi|\geq 1$.
The claim \eqref{e:SingClaim1} is equivalent to showing $\lim_{t\to 0}J_t=0$, where
\begin{align*}
J_t=\sup_s(1+|s|^{1+\alpha})\left|\int_{|\xi|<\tfrac{1}{t}}\frac{\partial_sf(s-t\xi)-\partial_sf(s)}{t\xi}\Psi_t\,d\xi+\int_{\R}\partial_s^2f(s)\Psi_0\,d\xi\right|.
\end{align*}
Let $\eps>0$ and fix $R>1$ so that $\int_{|\xi|>R}|\Psi_t|\,d\xi<\eps$ for all $t\geq 0$ (by the bound \eqref{e:DsPhi} this is possible). Moreover, fix $0<\delta<1$ so that
\begin{align*}
\sup_{s\in\R}(1+|s|^{1+\alpha})\left|\frac{\partial_sf(s)-\partial_sf(s-\eta)}{\eta}-\partial_s^2f(s)\right| <\eps
\end{align*}
for all $0<|\eta|<\delta$. This is possible if $f\in C^{2,\alpha}_*(\R)$ since we can write
\begin{align*}
\frac{\partial_sf(s)-\partial_sf(s-\eta)}{\eta}-\partial_s^2f(s)=\int_0^1[\partial_s^2f(s-\tau\eta)-\partial_s^2f(s)]\,d\tau.
\end{align*}
Then, for any $t>0$ with $tR<\delta$ we have
\begin{align*}
J_t&\leq \sup_s(1+|s|^{1+\alpha})\left|\int_{|\xi|<R}\frac{\partial_sf(s-t\xi)-\partial_sf(s)}{t\xi}\Psi_t+\partial_s^2f(s)\Psi_0\,d\xi\right|\\
&+\sup_s(1+|s|^{1+\alpha})\int_{|\xi|>R}\left|\frac{\partial_sf(s-t\xi)-\partial_sf(s)}{t\xi}\Psi_t\right|+\left|\partial_s^2f(s)\Psi_0\right|\,d\xi\\
&\leq \sup_s(1+|s|^{1+\alpha})\int_{|\xi|<R}\left|\frac{\partial_sf(s-t\xi)-\partial_sf(s)}{t\xi}+\partial_s^2f(s)\right||\Psi_t|\,d\xi\\
&+\|f\|_{2,\alpha}^*\int_{|\xi|<R}|\Psi_t-\Psi_0|\,d\xi+\|f\|_{2,\alpha}^*\int_{|\xi|>R}|\Psi_t|+|\Psi_0|\,d\xi\\
&\leq C(1+\|f\|_{2,\alpha}^*)\eps+\|f\|_{2,\alpha}^*\int_{|\xi|<R}|\Psi_t-\Psi_0|\,d\xi.
\end{align*}
Using once more the bound \eqref{e:DsPhi} we deduce that $\limsup_{t\to 0}J_t\leq C(1+\|f\|_{2,\alpha}^*)\eps$. Since $\eps>0$ was arbitrary, it follows that $\lim_{t\to 0}J_t=0$, concluding the proof of \eqref{e:SingClaim1} and thence the proof of \eqref{e:TDeltas}.
\end{proof}

\bigskip

Using Lemmas \ref{l:regular} and \ref{l:singular} we can calculate
\begin{align*}
\left.\frac{\partial}{\partial t}\right\vert_{t=0}\frac{u_\nu^++u_\nu^-}{2}&=\lim_{t\to 0}\frac{1}{t}\left(\frac{u_\nu^++u_\nu^-}{2}-u^0_\nu\right)\\
&=\lim_{t\to 0}\frac{1}{t}T_{\left(\tfrac{\Phi_++\Phi_-}{2}-\Phi_0\right)}z_0+T_{\left(\tfrac{\Phi_++\Phi_-}{2}\right)}\tfrac{z-z_0}{t}\\
&=T_{\Phi_1}\bar{z}_0+T_{\Phi_0}z_0'+c\partial_s^2z_0\sigma(s),
\end{align*}
where $\sigma$ is the function defined in Lemma \ref{l:singular}. 
This motivates our choice of $z(s,t)$: 

\begin{theorem}\label{t:z}
Assume that $z_0(s)=\beta s+\bar{z}_0(s)$ with $\bar{z}_0\in C^{3,\alpha}_*(\R)$ for some $0<\alpha<1$ and $\beta\in\R$. 	
Let
\begin{equation*}
	z(s,t)=z_0(s)+tz_1(s)+\tfrac{1}{2}t^2z_2(s),
\end{equation*}
where
\begin{align*}
	z_1&:=T_{\Phi_0}\bar{z}_0,\\
	z_2&:=T_{\Phi_0}z_1+T_{\Phi_1}\bar{z}_0+c\sigma\partial_s^2z_0,
\end{align*}
with $\sigma=\frac{1-(\partial_sz_0)^2}{(1+(\partial_sz_0)^2)^2}$.
Then $z(s,t)=\beta s+\bar{z}(s,t)$ with $\bar{z}\in C^2([0,\infty);C_*^{1,\alpha}(\R))$ and with this choice of $z$ the conditions \eqref{z1}-\eqref{z2} of Theorem \ref{Subs} are satisfied. 
\end{theorem}

\begin{proof}\hfill

\subsubsection*{Step 1: Estimating $z_1$ and $z_2$}

We begin by showing that $z_1\in C_*^{2,\alpha}(\R)$ and $z_2\in C_*^{1,\alpha}(\R)$. First of all, since $\partial_sz_0\in C^2(\R)$ with $\partial_s^2z_0\in C^{1,\alpha}_*(\R)$, it follows that $\sigma\in C^2(\R)$ with $\sigma,\partial_s\sigma,\partial_s^2\sigma\in L^{\infty}(\R)$ and consequently $c\sigma\partial_s^2z_0\in C^{1,\alpha}_*(\R)$. Therefore, according to estimate \eqref{e:TPhi2} in Theorem \ref{t:TPhi} it suffices to show that $\Phi_0\in \mathcal{W}^{2,\alpha}$ and $\Phi_1\in \mathcal{W}^{1,\alpha}$, where $\Phi_0$ and $\Phi_1$ are defined in \eqref{e:Phi0} and \eqref{e:Phi1} above. 

\bigskip

The proof that $\Phi_0\in\mathcal{W}^0$ follows entirely along the lines of the proof of Lemma \ref{l:Tpm}: Since $\lim_{|\xi|\to\infty}Z_0=\beta$ we see that $\Phi_0^\infty=\frac{2}{1+\beta^2}$. Note also
that
$$
\sup_{\xi,s}\left|\xi(Z_0-\beta)\right|=\sup_{\xi,s}\left|\bar{z}_0(s)-\bar{z}_0(s-\xi)\right|\leq 2\|\bar{z}_0\|_0.
$$
Therefore 
$$
\bar{\Phi}_0=2\xi\frac{\beta^2-Z_0^2}{(1+Z_0^2)(1+\beta^2)}=2\frac{\beta+Z_0}{(1+Z_0^2)(1+\beta^2)}(\xi(Z_0-\beta))
$$
is bounded uniformly in $\xi,s$. Similarly, we calculate
$$
\xi\partial_\xi\Phi_0=\frac{-4Z_0}{(1+Z_0^2)^2}\xi\partial_{\xi}Z_0
$$
and
$$
\xi\partial_{\xi}Z_0=\partial_sz_0(s-\xi)-\frac{z_0(s)-z_0(s-\xi)}{\xi}=\partial_s\bar{z}_0(s-\xi)-\frac{\bar{z}_0(s)-\bar{z}_0(s-\xi)}{\xi}.
$$
We deduce that also $\tilde\Phi_0$ is uniformly bounded. This shows that $\Phi_0\in\mathcal{W}^0$. 

Next, we calculate:
$$
\partial_s\Phi_0=\frac{-4Z_0\partial_sZ_0}{(1+Z_0^2)^2},\quad \partial_s^2\Phi_0=\frac{16Z_0^2(\partial_sZ_0)^2}{(1+Z_0^2)^3}-\frac{4(\partial_sZ_0)^2+4Z_0\partial_s^2Z_0}{(1+Z_0^2)^2}\,.
$$
Note that, since $\partial_sz_0\in C^{2,\alpha}(\R)$, the function $Z_0(\xi,s)$ satisfies $Z_0\in C^{2,\alpha}(\R^2)$ and this implies $\Phi_0\in C^{2,\alpha}(\R^2)$. Furthermore, proceeding as above, 
it follows easily that $\xi\partial_s^kZ_0$ is bounded uniformly for $s\in\R, |\xi|>1$ for $k=1,2$, and hence the same is true for $\xi\partial_s^k\Phi_0$. Analogously, $\xi\partial_\xi\partial_s^kZ_0$ is bounded uniformly for $s\in\R, |\xi|>1$ for $k=0,1,2$, hence the same is true for $\xi\partial_\xi\partial_s^k\Phi_0$. This implies that $\partial_s\Phi_0,\partial_s^2\Phi_0\in\mathcal{W}^0$.

In the same manner we deduce that for $k=1,2$ 
\begin{align*}
|\xi||\partial_s^kZ_0(\xi,s_1)-\partial_s^kZ_0(\xi,s_2)|\leq & |\partial_s^kz_0(s_1)-\partial_s^kz_0(s_2)|\\
&+|\partial_s^kz_0(s_1-\xi)-\partial_s^kz_0(s_2-\xi)|\\
\leq &2[\partial_s^kz_0]_{\alpha}|s_1-s_2|^{\alpha}
\end{align*}
and similarly, for $|\xi|>1$ and $k=1,2$
$$
|\xi||\partial_\xi\partial_s^kZ_0(\xi,s_1)-\partial_\xi\partial_s^kZ_0(\xi,s_2)|\leq  ([\partial_s^{k+1}z_0]_{\alpha}+2[\partial_s^kz_0]_{\alpha})|s_1-s_2|^{\alpha}
$$
This shows that $\Phi_0\in \mathcal{W}^{2,\alpha}$. 

We next proceed with $\Phi_1$. First of all, from Theorem \ref{t:TPhi} and the above we deduce that $z_1\in C_*^{2,\alpha}$. 
Comparing the expressions for $\Phi_1$ and $\partial_s\Phi_0$ we see that the estimates for $\Phi_1$ may be obtained exactly as above, by replacing $\partial_sz_0$ by $z_1$ where appropriate and noting that both functions are in $C^{2,\alpha}$. In this way we deduce that $\Phi_1\in \mathcal{W}^{1,\alpha}$.

\subsubsection*{Step 2: The estimate \eqref{z1}}

Using \eqref{e:u_nu}-\eqref{e:TPhi} we have
$$
u_\nu^\pm-\partial_tz=\left(T_{\Phi_\pm}\bar{z}_0+tT_{\Phi_\pm}z_1+\tfrac12t^2T_{\Phi_\pm}z_2\right)-(z_1+tz_2),
$$
whereas, by our choice of $z_1$
\begin{align*}
T_{\Phi_\pm}\bar{z}_0-z_1&=T_{\Phi_\pm}\bar{z}_0-T_{\Phi_0}\bar{z}_0.
\end{align*}
Since $\bar{z}_0,z_1,z_2\in C^{1,\alpha}_*(\R)$ by Step 1, Lemmas \ref{l:Tpm} and \ref{l:continuity} imply that 
$$
\lim_{t\to 0}\sup_s(1+|s|^{1+\alpha})|u_\nu^\pm(s,t)-\partial_tz(s,t)|=0.
$$
In particular \eqref{z1} follows.

\subsubsection*{Step 3: The estimate \eqref{z2}}
As in Step 2, using \eqref{e:u_nu}-\eqref{e:TPhi}, our choice of $z_1,z_2$ and the linearity of $\Phi\mapsto T_\Phi$, we have
\begin{align*}
\frac{1}{t}\left(\frac{u_\nu^++u_\nu^-}{2}-\partial_tz\right)&=\frac{1}{t}\left(T_{\tfrac{\Phi_++\Phi_-}{2}}\bar{z}_0-z_1\right)+T_{\tfrac{\Phi_++\Phi_-}{2}}z_1-z_2+\tfrac{t}{2}T_{\tfrac{\Phi_++\Phi_-}{2}}z_2\\
=&\left(\tfrac{1}{t}T_{\Delta_{reg}\Phi}\bar{z}_0-T_{\Phi_1}\bar{z}_0\right)+\left(\tfrac{1}{t}T_{\Delta_{sing}\Phi}\bar{z}_0-c\partial_s^2\bar{z}_0\sigma(s)\right)\\
&+\left(T_{\tfrac{\Phi_++\Phi_-}{2}}z_1-T_{\Phi_0}z_1\right)+\left(\tfrac{t}{2}T_{\tfrac{\Phi_++\Phi_-}{2}}z_2\right).
\end{align*}
Since $\bar{z}_0\in C^{2,\alpha}_*(\R)$ and $z_1,z_2\in C^{1,\alpha}_*(\R)$, Lemmas \ref{l:Tpm}, \ref{l:continuity}, \ref{l:regular} and \ref{l:singular} are applicable, and 
we conclude that 
$$
\lim_{t\to 0}\sup_s(1+|s|^{1+\alpha})\left|\frac{1}{t}\left(\frac{u_\nu^+(s,t)+u_\nu^-(s,t)}{2}-\partial_tz(s,t)\right)\right|=0.
$$
In particular \eqref{z2} follows.

\end{proof}

%%%%%%%%%%%%%%%%%%%%%%%%%%%%%%%%%%%%%%%%%%%%%%%%%%%%%%%%%%%%%%%%%%%%%%%%%%%%

\section{Symmetric piecewise constant densities}\label{s:general}

The subsolution (and corresponding admissible weak solutions) constructed in the previous sections have a mixing zone with a maximal expansion rate of 
$c_{max}=1$ in the unstable case, whereas the maximal expansion rate reachable with the construction in \cite{ccf:ipm} is $c_{max}=2$. In this section we show that
with a more general piecewise constant density the method of this paper is applicable to reach the expansion rate $c_{max}=2$ as well -- indeed, this is easily achieved by approximating the linear density function from \cite{ccf:ipm} by a piecewise constant density. From now on we restrict attention to the unstable case, with $\rho^+=1$ and $\rho^-=-1$.

Let $N \in \N$ and define $2N$ interfaces
\begin{equation}\label{e:Gammas}
\Gamma^{\pm i}= \{x \in \R^2| x_2 = z(x_1,t) \pm c_i t\},\qquad  i = 1,\ldots,N,
\end{equation}
where $0<c_1<c_2<\ldots < c_N$ are arbitrary velocities and $z(s,t)$ is the parametrization of a curve for $t\in [0,T]$. 
The open regions between neighbouring interfaces are defined as
\begin{equation}\label{e:densn}
\begin{split}
\Omega^0(t) &= \{x \in \R^2:\, -c_1t < x_2-z(x_1,t) < c_1t\}\,,\\
\Omega^{N}(t) &= \{x \in \R^2 :\, x_2>z(x_1,t) + c_N t\},\\
\Omega^{-N}(t) &= \{x \in \R^2 :\, x_2<z(x_1,t) - c_N t\},
\end{split}
\end{equation}
and for $i=1,\dots,N-1$
\begin{equation}\label{e:densni}
\begin{split}
\Omega^{i}(t) &= \{x \in \R^2:\, z(x_1,t)+c_{i}t < x_2 < z(x_1,t)+c_{i+1}t\},\\
\Omega^{-i}(t) &= \{x \in \R^2:\, z(x_1,t)-c_{i+1}t < x_2 < z(x_1,t)-c_{i}t\}.
\end{split}	
\end{equation}
Analogously to \eqref{e:dens} we set
\begin{equation}\label{rhon}
\rho(x,t) = \frac{i}{N} \qquad\textrm{ for } x \in \Omega^{i}(t),\quad i=-N,\dots,N,
\end{equation}
so that we have a constant density jump of $\frac{1}{N}$ across each boundary $\Gamma^{\pm i}$. The mixing zone is then 
$$
\Omega_{mix} = \bigcup\limits_{i=-(N-1)}^{N-1} \Omega^i=\{x \in \R^2| -c_Nt < x_2-z(x_1,t) < c_Nt\}.
$$ 

With the density defined in this way, the velocity $u$ can be obtained as in Section \ref{s:u}, in particular we have the expression for the normal velocity component:
\begin{align*}
u_{\nu}^{(i)}(s,t) &:= u\left(s,z(s,t)+c_it,t\right)\cdot\begin{pmatrix}-\partial_sz(s,t)\\ 1\end{pmatrix}\\
&=\sum_{j=1}^N\frac{1}{2\pi N}\textrm{PV}\int_{\R}\frac{\partial_sz(s-\xi,t)-\partial_sz(s,t)}{\xi}\Phi_{ij}(\xi,s,t)\,d\xi
\end{align*}
for $i=-N\dots N$, $i\neq 0$, where 
\begin{align*}
\Phi_{i,j}(\xi,s,t)=&\frac{\xi^2}{\xi^2+(z(s-\xi,t)-z(s,t)-(c_i-c_j)t)^2}\\
&+\frac{\xi^2}{\xi^2+(z(s-\xi,t)-z(s,t)-(c_i+c_j)t)^2}\,
\end{align*}
and we have set $c_{-i}=-c_i$ for $i=1,\dots,N$.

\begin{theorem}\label{Subsn}
Suppose that $z(s,t)=\beta s+\bar{z}(s,t)$ with $\bar{z}\in C^1([0,T);C_*^{1,\alpha}(\R))$ satisfies
\begin{align}
\lim_{t\to 0}&\frac{1}{t}\left\|\partial_t z(\cdot,t) - \sum_{i=1}^N\frac{u_\nu^{(i)}(\cdot,t) + u_\nu^{(-i)}(\cdot,t)}{2N}\right\|_{L^1} = 0, \label{z1n} \\
\lim_{t\to 0}&\left\|\partial_t z(\cdot,t) - u_\nu^{(\pm i)}(\cdot,t)\right\|_{L^{\infty}} =  0\,\textrm{ for all $i=1\dots N$.}\label{z2n}
\end{align}
Moreover, let
\begin{align}\label{e:veln}
0 < c_i < \frac{2i-1}{N}\,\quad i=1,\dots,N.
\end{align}
Then there exists $T_*\in(0,T)$ such that there exists an admissible subsolution for  \eqref{IPM1}-\eqref{IPM4} on $[0,T_*]$ with initial datum $\rho_0$ given by \eqref{init} in the unstable case with $z_0=z\vert_{t=0}$. Furthermore, the density of the subsolution satisfies \eqref{e:densn}-\eqref{rhon}.
\end{theorem}

\begin{proof}

We proceed as in the proof of Theorem \ref{Subs} and set
$$
m=\rho u-(1-\rho^2)(\gamma+\tfrac12 e_2),
$$
with 
$$
\gamma=\nabla^\perp g^{(i)}\,\textrm{  in }\Omega^i,\quad i=-N\dots N,
$$
where $g^{(N)}=g^{{(-N)}}=0$ and $g^{(i)}\in C^1(\overline{\Omega^i})$ for $i=-(N-1)\dots (N-1)$ are to be determined. Then, \eqref{str} amounts to the conditions
$$
|\nabla g^{(i)}|<\frac12\quad\textrm{ in }\Omega^i\quad i=-(N-1)\dots(N-1),
$$
and \eqref{IPMsub} reduces to jump conditions \eqref{e:jump} on each interface: for any $i=1,\dots,N$
\begin{align}
\partial_\tau g^{(\pm(i-1))}&=\frac{1}{1-(\frac{i-1}{N})^2}\left\{\frac{c_i}{N}-\frac{2i-1}{2N^2}+\left(1-(\tfrac{i}{N})^2\right)\partial_{\tau}g^{(\pm i)}\pm \frac{\partial_tz-u_\nu^{(\pm i)}}{N}\right\}\notag\\
&=h^{(\pm i)}+\frac{1-(\tfrac{i}{N})^2}{1-(\frac{i-1}{N})^2}\partial_{\tau} g^{(\pm i)}\quad\textrm{ on }\Gamma^{\pm i}\,,\label{e:defh}
\end{align}
with
$$
h^{(\pm i)}(s,t)=\frac{1}{1-(\frac{i-1}{N})^2}\left\{\frac{c_i}{N}-\frac{2i-1}{2N^2}\pm \frac{\partial_tz-u_\nu^{(\pm i)}}{N}\right\}
$$
and 
$$
\partial_\tau g(x,t)=\partial_{x_1}g(x,t)+\partial_{x_2}g(x,t)\partial_{x_1}z(x_1,t).
$$
Since $g^{(\pm N)}=0$, we may use this expression to inductively define $g^{(\pm (N-1))}$, $g^{(\pm (N-2))},\dots,g^{(\pm 1)}$ as
\begin{align}
g^{(\pm i)}(x_1,t)&=\int_0^{x_1} h^{(\pm (i+1))}(s',t)+\frac{1-(\tfrac{i+1}{N})^2}{1-(\frac{i}{N})^2}\partial_{\tau} g^{(\pm (i+1))}(s',t)\,ds'.\label{e:defgi}
\end{align}
Note that by our choice $g^{(\pm i)}$ for $i\neq 0$ is a function of $x_1,t$ only, therefore 
$$
\partial_{\tau} g^{(\pm i)}=\partial_{x_1} g^{(\pm i)}=h^{(\pm (i+1))}+\frac{1-(\tfrac{i+1}{N})^2}{1-(\frac{i}{N})^2}\partial_{x_1} g^{(\pm (i+1))}
$$
for $i=1,\dots,(N-1)$. Now, by \eqref{z2n} we have $\|\partial_tz-u_\nu^{\pm i}\|_{L^\infty}=o(1)$ as $t\to 0$ for all $i=1,\dots,N$, so that
$$
h^{(\pm (i+1))}=\frac{1}{1-(\frac{i}{N})^2}\left\{\frac{c_{i+1}}{N}-\frac{2i+1}{2N^2}\right\}+o(1)\,.
$$
Using inductively \eqref{e:defh} we then obtain
\begin{align*}
\partial_{x_1} g^{(\pm i)}&=\frac{1}{1-(\tfrac{i}{N})^2}\sum_{j=i+1}^N\left(\frac{c_j}{N}-\frac{2j-1}{2N^2}\right)+o(1)\,.
\end{align*}
From \eqref{e:veln} we have $|\frac{c_j}{N}-\frac{2j-1}{2N^2}|<\frac{2j-1}{2N^2}$, so that
$$
\left| \sum_{j=i+1}^N\left(\frac{c_j}{N}-\frac{2j-1}{2N^2}\right)\right|\leq \sum_{j=i+1}^N\frac{2j-1}{2N^2}=\frac{N^2-i^2}{2N^2}\,.
$$
We deduce that for $i=1,\dots,N$ we have $\partial_{x_2}g^{(\pm i)}=0$ and
\begin{equation}\label{e:g1i}
\left\|\partial_{x_1}g^{(\pm i)}(\cdot,t)-\tfrac{1}{2}\right\|_{L^{\infty}}\to 0\,\textrm{ as }t\to 0.
\end{equation}
It remains to construct $g^{(0)}$ in $\Omega^0$. As in the proof of Theorem \ref{Subs}
we define for $s\in\R$, $t\in (0,T)$ and $\lambda\in [-c_1t,c_1t]$ the function
$$
\hat{g}(s,\lambda,t):=g^{(0)}(s,z(s,t)+\lambda,t)
$$
and set, in accordance with \eqref{e:defh} with $i=1$,
\begin{align*}
\hat{g}&(s,\lambda,t) =  \frac{c_1t+\lambda}{2c_1t} \left( \int_0^{s} (h^{(+1)}+(1-\tfrac{1}{N^2})\partial_{x_1}g^{(+1)}) ds' \right) \\
&+ \frac{c_1t-\lambda}{2c_1t} \left(\int_0^{s} (h^{(-1)}+(1-\tfrac{1}{N^2})\partial_{x_1}g^{(-1)})ds' \right)\\
=&  s\left(\frac{c_1}{N}-\frac{1}{2N^2}\right)\\
&+\frac{\lambda+c_1t}{2c_1tN} \left( \int_0^{s} \partial_tz-u_\nu^{(+1)} ds' \right)+\frac{\lambda-c_1t}{2c_1tN} \left( \int_0^{s} \partial_tz-u_\nu^{(-1)} ds' \right) \\
&+ \left(1-\frac{1}{N^2}\right)\left(\frac{c_1t+\lambda}{2c_1t} g^{(+1)}(s,t)+\frac{c_1t-\lambda}{2c_1t} g^{(-1)}(s,t) \right).
\end{align*}
Using \eqref{z2n}, \eqref{e:veln} and \eqref{e:g1i} we deduce that
\begin{equation}\label{e:g2i}
\left\|\partial_{x_1}g^{(0)}(\cdot,t)-\tfrac{1}{2}\right\|_{L^{\infty}}\to 0\,\textrm{ as }t\to 0.
\end{equation}
Furthermore, we have
\begin{align*}
\partial_{\lambda}\hat{g}=&\frac{1}{2c_1tN}\int_0^s2\partial_tz-u_\nu^{(+1)}-u_\nu^{(-1)}\,ds'+\frac{1}{2c_1t}\left(1-\frac{1}{N^2}\right)\left(g^{(+1)}-g^{(-1)}\right).
\end{align*}
Now, from \eqref{e:defh} and the choice $g^{(\pm N)}=0$ we obtain
$$
\partial_{x_1}g^{(1)}-\partial_{x_1}g^{(-1)}=\frac{N}{N^2-1}\sum_{j=2}^N2\partial_tz-u_{\nu}^{(+j)}-u_{\nu}^{(-j)},
$$
so that
\begin{align*}
\partial_{\lambda}\hat{g}=\frac{1}{c_1t}\int_0^s\left(\partial_tz-\sum_{j=1}^N\frac{u_\nu^{(+j)}+u_\nu^{(-j)}}{2N}\right)\,ds'.
\end{align*}
In particular, since $\partial_\lambda\hat{g}=\partial_{x_2}g^{(0)}$, it follows from \eqref{z1n} that
\begin{equation}\label{e:g3i}
\left\|\partial_{x_2}g^{(0)}(\cdot,t)\right\|_{L^{\infty}}\to 0\,\textrm{ as }t\to 0.
\end{equation}
From \eqref{e:g1i}, \eqref{e:g2i} and \eqref{e:g3i} Êwe finally deduce that
$$
|\nabla g^{(i)}|<\frac{1}{2}\textrm{ for all $i$ for sufficiently small }t>0.
$$
This concludes the proof.
\end{proof}

\bigskip

We now construct a curve $z=z(s,t)$ satisfying \eqref{z1n}-\eqref{z2n} analogously to the construction in Section \ref{s:z}. Indeed, assume that $z(s,t)=\beta s+\bar{z}(s,t)$ with $\bar{z}\in C^1([0,T];C^{1,\alpha}_*(\R))$ for some $0<\alpha<1$ and $\beta\in\R$.  
Recall that 
$$
u_{\nu}^{(\pm i)}=\frac{1}{N}\sum_{j=1}^NT_{\Phi_{\pm i,j}}z
$$
for $i=1,\dots,N$. Note that for any $i,j$
$$
\left.\Phi_{\pm i,j}\right|_{t=0}=\frac{2\xi^2}{\xi^2+(z_0(s-\xi)-z_0(s))^2}=\Phi_0,
$$
where $\Phi_0$ is defined in \eqref{e:Phi0}.
Lemma \ref{l:continuity} applies to show that
$$
\lim_{t\to 0}\sup_{s\in\R}(1+|s|^{1+\alpha})\left|T_{\Phi_{\pm i,j}}f(s)-T_{\Phi_0}f(s)\right|=0
$$
whenever $f\in C^{1,\alpha}_*(\R)$. We deduce
\begin{equation}\label{e:manycurves1}
\lim_{t\to 0}\sup_{s\in\R}(1+|s|^{1+\alpha})\left|u_{\nu}^{(\pm i)}(s)-T_{\Phi_0}\bar{z}_0(s)\right|=0,
\end{equation}
where $\bar{z}_0=\bar{z}|_{t=0}$. 
Next, we calculate $\left.\tfrac{\partial}{\partial t}\right|_{t=0}\tfrac{u_{\nu}^{(i)}+u_{\nu}^{(-i)}}{2}$. To this end set, analogously to Section \ref{s:z},
$$
\Delta\Phi_{i,j}=\frac{\Phi_{i,j}+\Phi_{-i,j}}{2}-\Phi_0=\Delta_{reg}\Phi_{ij}+\Delta_{sing}\Phi_{ij}\,
$$
where
\begin{equation*}
\begin{split}
\Delta_{reg}\Phi_{ij}=\frac{Z_0^2-Z_t^2}{2(1+Z_0^2)}&\left(\frac{1}{1+(Z_t+(c_i-c_j)\tfrac{t}{\xi})^2}+\frac{1}{1+(Z_t-(c_i-c_j)\tfrac{t}{\xi})^2}\right.\\
&+\left.\frac{1}{1+(Z_t+(c_i+c_j)\tfrac{t}{\xi})^2}+\frac{1}{1+(Z_t-(c_i+c_j)\tfrac{t}{\xi})^2}\right)
\end{split}
\end{equation*}
and
\begin{equation*}
\begin{split}
\Delta_{sing}\Phi_{ij}=&\frac{1}{1+Z_0^2}\left(\frac{-(c_i-c_j)tZ_t\xi-\tfrac{(c_i-c_j)^2}{2}t^2}{\xi^2+(Z_t\xi+(c_i-c_j)t)^2}+\frac{(c_i-c_j)tZ_t\xi-\tfrac{(c_i-c_j)^2}{2}t^2}{\xi^2+(Z_t\xi-(c_i-c_j)t)^2}\right.\\
&+\left.\frac{-(c_i+c_j)tZ_t\xi-\tfrac{(c_i+c_j)^2}{2}t^2}{\xi^2+(Z_t\xi+(c_i+c_j)t)^2}+\frac{(c_i+c_j)tZ_t\xi-\tfrac{(c_i+c_j)^2}{2}t^2}{\xi^2+(Z_t\xi-(c_i+c_j)t)^2}\right).
\end{split}
\end{equation*}
As in Lemma \ref{l:regular}, we have 
$$
\lim_{t\to 0}\sup_{s\in\R}(1+|s|^{1+\alpha})\left|\tfrac{1}{t}T_{\Delta_{reg}\Phi_{ij}}f(s)-T_{\Phi_1}f(s)\right|=0
$$
whenever $f\in C^{1,\alpha}_*(\R)$ and, using Lemma \ref{l:singular},  
$$
\lim_{t\to 0}\sup_{s\in\R}(1+|s|^{1+\alpha})\left|\tfrac{1}{t}T_{\Delta_{sing}\Phi_{ij}}f(s)-c_{ij}\partial_s^2f(s)\sigma(s)\right|=0
$$
whenever $f\in C^{2,\alpha}_*(\R)$, where $c_{ij}=\tfrac{1}{2}|c_i-c_j|+\tfrac{1}{2}|c_i+c_j|=\max(c_i,c_j)$.
Consequently, 
\begin{equation}\label{e:manycurves2}
\left.\frac{\partial}{\partial t}\right|_{t=0}\frac{1}{N}\sum_{i=1}^N\tfrac{u_{\nu}^{(i)}+u_{\nu}^{(-i)}}{2}=T_{\Phi_1}\bar{z}_0+T_{\Phi_0}z_0'+\bar{c}\partial_s^2z_0\sigma,
\end{equation}
where
$$
\bar{c}=\frac{1}{N^2}\sum_{i,j=1}^N\max(c_i,c_j),
$$
and $z_0'$ and $\Phi_1$ are defined in \eqref{e:Phi1} in Section \ref{s:z}. From these considerations we deduce

\begin{theorem}\label{t:zN}
Assume that $z_0(s)=\beta s+\bar{z}_0(s)$ with $\bar{z}_0\in C^{3,\alpha}_*(\R)$ for some $0<\alpha<1$ and $\beta\in\R$. 	
Let
\begin{equation*}
	z(s,t)=z_0(s)+tz_1(s)+\tfrac{1}{2}t^2z_2(s),
\end{equation*}
where
\begin{align*}
	z_1&:=T_{\Phi_0}\bar{z}_0,\\
	z_2&:=T_{\Phi_0}z_1+T_{\Phi_1}\bar{z}_0+\bar{c}\sigma\partial_s^2z_0,
\end{align*}
with $\sigma=\frac{1-(\partial_sz_0)^2}{(1+(\partial_sz_0)^2)^2}$ and $\bar{c}=\frac{1}{N^2}\sum_{i,j=1}^N\max(c_i,c_j)$.
Then $z(s,t)=\beta s+\bar{z}(s,t)$ with $\bar{z}\in C^2([0,\infty);C_*^{1,\alpha}(\R))$ and with this choice of $z$ the conditions \eqref{z1n}-\eqref{z2n} of Theorem \ref{Subsn} are satisfied. 
\end{theorem}

The proof is entirely analogous to the proof of Theorem \ref{t:z}, based on the above calculations and Lemmas \ref{l:continuity}, \ref{l:regular} and \ref{l:singular}.

Finally, observe that for the subsolution obtained in Theorem \ref{Subsn} the rate of expansion of the mixing zone is given by $c_n<\frac{2n-1}{n}=2-\frac{1}{n}$, so that any expansion rate $c<2$ is obtainable by choosing $n$ sufficiently large. 

\bibliographystyle{acm}
%\bibliography{leipzig}

%\nocite{alv:nfpm}
%\nocite{cgvs:gr2dm}
%\nocite{cgo:abipm}

\end{document}